\documentclass[reqno]{amsart}
\usepackage[utf8x]{inputenc}  

\usepackage{setspace}
\usepackage[left=3cm, right=3cm, bottom=4cm]{geometry}                 

\usepackage{graphicx} 
\usepackage{pgf,tikz}
\usetikzlibrary{arrows}
\usepackage{amssymb}
\usepackage{pdfsync}
\usepackage{mathrsfs}
\usepackage{hyperref} 
\usepackage{verbatim} 
\usepackage{epstopdf}
\usepackage{bbm} 
\usepackage[colorinlistoftodos,prependcaption,textsize=tiny]{todonotes}
\usepackage{xargs}

\def\R{\mathbb{R}}

\def\N{\mathbb{N}}
\def\Z{\mathbb{Z}}

\def\f{\widehat{f}}

\def\gs {{g^\#}}

\renewcommand{\d}{\text{\rm d}}

\def\x {{\widetilde{x}}}
\def\y {{\widetilde{y}}}
\def\sgn {{\rm sgn}}

\newtheorem{theorem}{Theorem}
\newtheorem*{theoremA}{Theorem A}
\newtheorem*{theoremB}{Theorem B}
\newtheorem{corollary}[theorem]{Corollary}

\newtheorem*{definition*}{Definition}

\newtheorem{proposition}[theorem]{Proposition}
\newtheorem{lemma}[theorem]{Lemma}

\makeatletter
\DeclareFontFamily{U}{tipa}{}
\DeclareFontShape{U}{tipa}{m}{n}{<->tipa10}{}
\newcommand{\arc@char}{{\usefont{U}{tipa}{m}{n}\symbol{62}}}%

\makeatother

\numberwithin{equation}{section}

\allowdisplaybreaks

\newcommand{\intav}[1]{\mathchoice {\mathop{\vrule width 6pt height 3 pt depth  -2.5pt
\kern -8pt \intop}\nolimits_{\kern -6pt#1}} {\mathop{\vrule width
5pt height 3  pt depth -2.6pt \kern -6pt \intop}\nolimits_{#1}}
{\mathop{\vrule width 5pt height 3 pt depth -2.6pt \kern -6pt
\intop}\nolimits_{#1}} {\mathop{\vrule width 5pt height 3 pt depth
-2.6pt \kern -6pt \intop}\nolimits_{#1}}}

\newcommand{\intavl}[1]{\mathchoice {\mathop{\vrule width 6pt height 3 pt depth  -2.5pt
\kern -8pt \intop}\limits_{\kern -6pt#1}} {\mathop{\vrule width 5pt
height 3  pt depth -2.6pt \kern -6pt \intop}\nolimits_{#1}}
{\mathop{\vrule width 5pt height 3 pt depth -2.6pt \kern -6pt
\intop}\nolimits_{#1}} {\mathop{\vrule width 5pt height 3 pt depth
-2.6pt \kern -6pt \intop}\nolimits_{#1}}}

\title[Generalized sign Fourier uncertainty]{Generalized sign Fourier uncertainty}

\author[Carneiro]{Emanuel Carneiro}
\author[Quesada-Herrera]{Emily Quesada-Herrera}

\address{
ICTP - The Abdus Salam International Centre for Theoretical Physics\\
Strada Costiera, 11, I - 34151, Trieste, Italy.}

\email{carneiro@ictp.it}

\address{Graz University of Technology, Institute of Analysis and Number Theory, Steyrergasse 30/II, 8010 Graz, Austria}
\email{quesada@math.tugraz.at}

\date{\today}                                           
\begin{document}

\subjclass[2010]{42B10}
\keywords{Fourier transform, sign uncertainty, harmonic polynomials, Bochner's identity, sharp constants.}
\begin{abstract} 
We consider a generalized version of the sign uncertainty principle for the Fourier transform, first proposed by Bourgain, Clozel and Kahane \cite{BCK} and revisited by Cohn and Gonçalves \cite{CG}. In our setup, the signs of a function and its Fourier transform resonate with a generic given function $P$ outside of a ball. One essentially wants to know if and how soon this resonance can happen, when facing a suitable competing weighted integral condition. The original version of the problem corresponds to the case $P = {\bf 1}$. Surprisingly, even in such a rough setup, we are able to identify sharp constants in some cases. 
\end{abstract}

\maketitle 

\section{Introduction}

\subsection{Background} Define the Fourier transform of a function $f\in L^1(\R^d)$ by 
\begin{equation*}
    \mathcal{F}_d[f](\xi)=\widehat{f}(\xi)=\int_{\R^d}e^{-2\pi i x \cdot \xi}\, f(x)\,\d x.
\end{equation*}
This transform is certainly one of the most fundamental objects in mathematics and applied mathematics, as it is used to model a variety of oscillatory phenomena. The expression {\it Fourier uncertainty} appears recurrently in the literature, describing many qualitative and quantitative variants of the same underlying principle: that one cannot have an unrestricted control of a function and its Fourier transform simultaneously. See \cite{{BDsurvey}, FSsurvey} for surveys on uncertainty principles.

\smallskip

The uncertainty paradigm is directly related to different sorts of Fourier optimization problems. Generically speaking, these are  problems in which one imposes suitable conditions on a function and its Fourier transform, and seeks to optimize a certain quantity of interest. There are surprising applications of such problems, for instance, in the theory of the Riemann zeta-function \cite{CCM, CChi, CS}, in bounding prime gaps \cite{CMS} and in the theory of sphere packings \cite{CE, CKMRV, Vi}.

\smallskip

A classical version of Heisenberg's uncertainty principle establishes that a function and its Fourier transform cannot simultaneously have their mass arbitrarily concentrated near the origin. This can be mathematically formulated as (see, for instance, \cite[Corollary 2.8]{FSsurvey})
\begin{equation}\label{20200324_00:22}
    ||f||_2^2\le \frac{4\pi}{d} \big\| |x|f \big\|_2\cdot \big\| |\xi|\f \big\|_2\,,
\end{equation}
for any $f\in L^2(\R^d)$. One may ask what happens if, instead of the total mass, one considers the {\it concentration of negative mass} of a function and its Fourier transform near the origin. In \cite{BCK}, Bourgain, Clozel and Kahane introduced a novel uncertainty principle that addresses this question, in connection to the study of real zeros of zeta functions over number fields and bounds for the associated discriminants. In their setup, the trade-off is between the sign of a function at infinity (or more precisely, the last sign change of the function), and a competing local condition for the transform at the origin. This uncertainty principle was later quantitatively refined by Gon\c{c}alves, Oliveira e Silva and Steinerberger in \cite{GOS}, who also studied its extremizers. More recently, Cohn and Gon\c{c}alves \cite{CG} went further in the topic, building on the fact that the original uncertainty principle of Bourgain, Clozel and Kahane \cite{BCK} is suitably associated to eigenfunctions of the Fourier transform with eigenvalue $+1$, by posing an analogous principle associated to the eigenvalue $-1$.

\smallskip

The sign uncertainty principles of \cite{BCK, CG, GOS} can be formulated as follows. We say that a measurable function $f:\R^d \to \R$ is {\it eventually non-negative}\footnote{It will be convenient here not to consider only continuous functions in the definition of {\it eventual non-negativity}, as the previous works in the literature do. Note, however, that we require that $f$ has a non-negative sign {\it for all} $|x| > r(f)$, and not only almost everywhere with respect to the Lebesgue measure. Similarly, we may define the concepts {\it ``eventually non-positive"} and {\it ``eventually zero"}.} if $f(x) \geq 0$ for all sufficiently large $|x|$, and we define
\begin{equation*}
r(f) :=\inf\{ r>0 \ : \ f(x) \geq 0 \ {\rm for \ all} \ |x|\geq r\}.
\end{equation*}
Let $s\in \{+1,-1\}$ denote a sign, and consider the following family of functions:
\begin{equation}\label{20200410_17:03}
\mathcal{A}_{s}(d) = \left\{
                \begin{array}{ll}
                  f\in L^1(\R^d)\setminus\{{\bf 0}\}\text{ continuous, even, real-valued and such that}\  \f\in L^1(\R^d);\\
                  sf(0)\le 0,\  \f(0)\le 0;\\
                  f\ {\rm and}\ s\f \text{ are eventually non-negative}.
                \end{array}
              \right\}.
\end{equation}
We then define 
\begin{equation}\label{20200317_13:30}
\mathbb{A}_s(d) := \inf_{f\in\mathcal{A}_s(d)}\sqrt{r(f) \, .\, r\big(s\f\,\big)}\,,
\end{equation}
which turns out to be a natural object of interest since $r(f)\, . \, r\big(s\f\,\big)$ is invariant under rescalings of the function $f$. The following assertion holds.

\begin{theoremA}[Bourgain, Clozel and Kahane \cite{BCK} ($s=+1$) ; Cohn and Gon\c{c}alves \cite{CG} ($s=-1$)] 
Let $s\in \{+1,-1\}$. Then there exist strictly positive universal constants $c$ and $C$ such that 
\begin{equation}\label{20200326_09:44}
c \sqrt{d} \leq \mathbb{A}_s(d) \leq C \sqrt{d}.
\end{equation}
\end{theoremA}
\noindent In particular note that $\mathbb{A}_s(d) >0$. Quantitatively speaking, from \cite{BCK, CG, GOS}, estimate \eqref{20200326_09:44} holds with $c = (2 \pi e)^{-1/2}$ for $s =\pm1$; $C = (2 \pi)^{-1/2} + o_d(1)$ for $s = +1$; and $C = 0.3194\ldots + o_d(1)$ for $s=-1$. An important step in the proof of Theorem A is the fact that one can reduce the search for the infimum in \eqref{20200317_13:30} to a restricted class $\mathcal{A}^{**}_s(d) \subset \mathcal{A}_{s}(d)$ given by
\begin{equation*}
\mathcal{A}^{**}_s(d) = \left\{
                \begin{array}{ll}
                  f\in L^1(\R^d)\setminus\{{\bf 0}\}\text{ continuous, radial and real-valued: }\f=sf;\\
                  f(0)= 0;\\
                  f \text{ is eventually non-negative}.
                \end{array}
              \right\}.
\end{equation*}
This is how the eigenfunctions of the Fourier transform appear in connection to these problems. The existence of extremizers for $\mathbb{A}_s(d)$  (i.e. functions that realize the infimum in \eqref{20200317_13:30}) in this restricted class was established in \cite{GOS} for $s=+1$ and in \cite{CG} for $s=-1$.

\smallskip

The exact values of $\mathbb{A}_s(d)$ are only known in four particular cases, discovered in some of the most influential works at the interface between analysis and number theory over the last years. Firstly, the celebrated works on the sphere packing problem via linear programming bounds \cite{CE, CKMRV, Vi} yield the sharp versions of the $(-1)$-uncertainty principle in dimensions $d=1,8$ and $24$ as corollaries (see the extended remark at the end of this subsection for the precise connection). In these cases, the optimal lower bound can be established via the classical Poisson summation formula (for the $E_8$-lattice in dimension $d=8$, and for the Leech lattice in dimension $d=24$). The formula then hints on the appropriate interpolating conditions  of the extremal functions. In dimension $d=1$, the function $f(x) = \sin^2(\pi x) / (x^2-1)$ is a bandlimited extremizer; see also the earlier work of Logan \cite{Lo}. In each of the dimensions $d=8$ and $24$, a radial Schwartz extremal eigenfunction (with prescribed values for the function and its radial derivative at the radii $\{\sqrt{2n}\, ;\,  n \in \N\})$ is constructed via the impressive machinery introduced by Viazovska \cite{Vi} on Laplace transforms of modular forms; see also the recent work \cite{CKMRV2}. Secondly, the recent work of Cohn and Gon\c{c}alves \cite{CG} establishes the sharp version of the $(+1)$-uncertainty principle of Bourgain, Clozel and Kahane in the special dimension $d=12$, where the optimal lower bound now comes from a Poisson summation formula for radial Schwartz functions on $\R^{12}$ derived from the Eisenstein series $E_6$, and an explicit radial Schwartz extremal eigenfunction is constructed by further exploring the ideas of Viazovska \cite{Vi}. We now compile such results.

\begin{theoremB} Let $s\in \{+1,-1\}$ and let $\mathbb{A}_s(d)$ be defined by \eqref{20200317_13:30}. Then
\begin{enumerate}
\item[(i)] {\rm (Corollaries of Cohn and Elkies \cite{CE} ($d=1$), Viazovska \cite{Vi} ($d=8$) and Cohn, Kumar, Miller, Radchenko and Viazovska \cite{CKMRV} ($d=24$)) }.
\begin{align}\label{20200327_14:00}
\mathbb{A}_{-1}(1) = 1 \ ; \  \mathbb{A}_{-1}(8) = \sqrt{2}  \ ; \ \mathbb{A}_{-1}(24) = 2.
\end{align}
\item[(ii)] {\rm (Cohn and Gon\c{c}alves \cite{CG})}.
\begin{align}\label{20200327_14:01}
\mathbb{A}_{+1}(12) = \sqrt{2}.
\end{align}
\end{enumerate}
\end{theoremB}

It is not known in general whether the search for the infimum in \eqref{20200317_13:30} can be restricted to Schwartz functions. This is only known to be true in the cases of Theorem B and in the additional case $(s, d) = (+1,1)$, recently established in \cite{GOR2}. It is conjectured that $\mathbb{A}_{-1}(2) = (4/3)^{1/4}$ and that $\mathbb{A}_{+1}(1) = (2\varphi)^{-1/2}$, where $\varphi = (1+\sqrt{5})/2$ is the golden ratio; see \cite[Conjectures 1.6 and 1.7]{GOR1}. The recent work \cite{GOR1} considers extensions of the $(\pm1)$-sign uncertainty principles to a more abstract operator setting, with very interesting applications to Fourier series and spherical harmonics, among others, and it will have important connections to the present paper. In a nutshell, this is the state of the art in this problem. 

\smallskip

A natural question that arises is the following: would there be suitable formulations of the sign uncertainty principle associated to the remaining eigenvalues $\pm i$? This was one of the original motivations for this work and, as we shall see, it will drive us to more general versions of such principles in the Euclidean space.

\smallskip

\noindent {\sc Remark:} (Connection between sign Fourier uncertainty and sphere packing). In \cite[Theorem 3.2]{CE}, Cohn and Elkies considered the following Fourier optimization problem, now regarded as the linear programming bound for the sphere packing problem. Consider the class
\begin{equation*}
\mathcal{A}_{LP}(d) = \left\{
                \begin{array}{ll}
                  g\in L^1(\R^d)\setminus\{{\bf 0}\}\text{ continuous, even, real-valued and such that}\  \widehat{g}\in L^1(\R^d);\\
                  g(0) = \widehat{g}(0) = 1;\\
                  -g \ \text{is eventually non-negative};\\
                  \widehat{g} \ \text{is non-negative};
                \end{array}
              \right\},
\end{equation*}
and define
\begin{equation*}
\mathbb{A}_{LP}(d) := \inf_{g\in\mathcal{A}_{LP}(d)} r(-g).
\end{equation*}
They showed that, given any sphere packing $\mathcal{P} \subset \R^d$ of congruent balls, its {\it upper density} $\Delta(\mathcal{P})$ (i.e. the fraction of the space covered by the balls in the packing; see \cite[Appendix A]{CE} for details) satisfies
$$\Delta(\mathcal{P}) \leq \mathbb{A}_{LP}(d)^d \, |B_{\frac12}|.$$
Numerical experiments suggested that this bound was sharp in dimensions $d=1,2,8$ and $24$, the latter three for the honeycomb, $E_8$ and Leech lattices, respectively. It was already pointed out in \cite{CE} that $\mathbb{A}_{LP}(1)=1$. In \cite{CKMRV}, Viazovska found the extremal function to show that $\mathbb{A}_{LP}(8) = \sqrt{2}$, hence establishing the optimality of the $E_8$-lattice in $d=8$. Later, Cohn, Kumar, Miller, Radchenko and Viazovska \cite{CKMRV} found the extremal function to show that $\mathbb{A}_{LP}(24) = 2$ and established the optimality of the Leech lattice in $d=24$. It is a classical theorem, proved by other methods, that the honeycomb lattice is optimal if $d=2$; see e.g. \cite{Hales}. Hence, it is conjectured that $\mathbb{A}_{LP}(2) = (4/3)^{1/4}$, but the corresponding extremal function has not yet been discovered. The connection between the $(-1)$-uncertainty principle and the linear programming bound is simple: if $g \in \mathcal{A}_{LP}(d)$ observe that $f:= \widehat{g} - g \in \mathcal{A}_{-1}(d)$ and that $r(f) \leq r(-g)$. Hence, one plainly has $\mathbb{A}_{-1}(d) \leq \mathbb{A}_{LP}(d)$. It is conjectured that, in fact, one has $\mathbb{A}_{-1}(d) = \mathbb{A}_{LP}(d)$ for all $d \geq 1$ (see \cite[Conjecture 8.2]{CE}, \cite{CG} and \cite{GOR1}) but, so far, this has only been established in the cases given by \eqref{20200327_14:00}.

\subsection{Generalized sign Fourier uncertainty} In what follows we write $x = (x_1, x_2, \ldots, x_d) \in \R^d$ for our generic variable (from now on used for both $f$ and $\f$). Related to \eqref{20200324_00:22}, there exist Heisenberg-type principles in the literature that say that $f$ and $\f$ cannot be simultaneously concentrated around the zero set of a function $Q:\R^d \to \R$. For instance, when $Q$ is a non-degenerate quadratic form on $\R^d$, a corollary of a theorem of Shubin, Vakilian and Wolff \cite{SVW} (see also \cite[Corollary 2.20]{BDsurvey}) establishes 
\begin{equation}\label{20200324_00:231}
    ||f||_2^2\le C \, \big\| Q f \big\|_2\cdot \big\| Q \f \big\|_2
\end{equation}
for $f \in L^2(\R^d)$, while Demange \cite{dem} establishes \eqref{20200324_00:231} when $Q(x)=|x_1|^{\gamma_1}|x_2|^{\gamma_2}\ldots|x_d|^{\gamma_d}$ with $\gamma_j>0$ for $1\le j\le d$. In a vague analogy to such results, we now consider a situation where the signs of $f$ and $\f$ at infinity are prescribed by a given generic function $P$ that we now describe.

\smallskip

Throughout the paper we let $P:\R^d \to \R$ be a measurable function, not identically zero on $\R^d \setminus \{0\}$, verifying:
\begin{itemize}
\item[(P1)] $P  \in L^1_{\rm loc}(\R^d)$.

\smallskip

\item[(P2)] $P$ is either even or odd. We let $\frak{r} \in \{0,1\}$ be such that 
\begin{equation}\label{20200404_03:06}
P(-x) = (-1)^{\frak{r}} P(x)
\end{equation}
for all $x \in \R^d$.\
\end{itemize}
We shall also consider the following pool of additional assumptions. In each of our results below, an appropriate subset of these may be required.

\begin{itemize}
\item[(P3)] $P$ is annihilating in the following sense: if $f \in L^1(\R^d)$ is a continuous eigenfunction of the Fourier transform such that $Pf$ is eventually zero then $f = {\bf  0}$.

\smallskip

\item[(P4)]  $P$ is homogeneous. That is, there is a real number $\gamma > -d$ such that 
\begin{equation}\label{20200401_00:18}
P(\delta x) = \delta^{\gamma} P(x)
\end{equation}
for all $\delta >0$ and $x \in \R^d$.

\smallskip

\item[(P5)] The sub-level set $A_{\lambda} = \{x \in \R^d \,:\, |P(x)|\leq \lambda\}$ has finite Lebesgue measure for some $\lambda >0$.

\smallskip

\item[(P6)] The sub-level set $A_{\lambda} = \{x \in \R^d \,:\, |P(x)|\leq \lambda\}$ is bounded for some $\lambda >0$.

\smallskip

\item[(P7)] $P \in L^{\infty}_{{\rm loc}}(\R^d)$.

\smallskip

\item[(P8)] $P\,e^{-\lambda \pi |\cdot |^2} \in L^1(\R^d)$ for all $\lambda >0$.

\smallskip

\item[(P9)] (Sign density) For each $x \in \R^d\setminus\{0\}$ such that $P(x) \neq 0$ we have 
\begin{equation*}
\liminf_{\varepsilon \to 0} \frac{\big| \{y \in \R^d :  P(y) P(x) > 0\} \cap B_{\varepsilon}(x)\big|}{\big|B_{\varepsilon}(x)\big|} >0.
\end{equation*}
\end{itemize}

\smallskip

\noindent{\sc Remark:} Condition (P3) above holds in a variety of situations. A simple one would be if the set $\{x \in \R^d \,:\, P(x) \neq 0\}$ is dense in $\R^d$  (in this case, $Pf$ eventually zero implies that $f$ has compact support). Another one is if the set $\{x \in \R^d \,:\, P(x) = 0\}$ has finite Lebesgue measure (hence (P6) implies (P5) that implies (P3)). In this case, $Pf$ eventually zero implies that $f$ is supported on a set of finite measure, and hence $f = {\bf  0}$ by Lemma \ref{AB_unc} below. Note also that (P1) and (P4) imply (P8).

\smallskip

We investigate the sign uncertainty principles in a more general setting as follows. In our formulation, it will be convenient to think of the competing conditions at the origin as weighted integrals over $\R^d$, via the Fourier transform. In this sense, the conditions $sf(0)\le 0$ and  $\f(0)\le 0$ appearing in \eqref{20200410_17:03} should be viewed as $\int_{\R^d}s\f \leq 0$ and $\int_{\R^d} f \leq 0$, respectively. Assume that our function $P$ verifies properties (P1), (P2), (P3) and (P4) above and let $s\in \{+1,-1\}$ be a sign. Consider the following class of functions, with suitable parity and integrability conditions (note that we move to a slightly different notation to denote the dependence on the function $P$),
\begin{equation}\label{20200510_10:23}
\!\!\!\!\! \mathcal{A}_{s}(P;d) = \!\left\{\!\!
                \begin{array}{ll}
                  f\in L^1(\R^d)\setminus\{{\bf 0}\}\text{ continuous, real-valued and such that } f(-x)= (-1)^\frak{r} f(x);\\
                 \f, Pf, P\f \in L^1(\R^d);\\
                  \int_{\R^d} Pf \le 0\ , \  \int_{\R^d} s(-i)^\frak{r}P\f \le 0;\\
                  P f ,\  s (-i)^\frak{r} P\f \text{ are eventually non-negative}.
                \end{array}
            \!\!\!\! \right\}\!.
\end{equation}
As before, let us define 
\begin{equation}\label{20200411_12:04}
\mathbb{A}_{s}(P;d) = \inf_{f\in\mathcal{A}_{s}(P;d) }\sqrt{r(Pf)\,.\,r\big( s (-i)^\frak{r}P\f \, \big)}.
\end{equation}
Note that if $f \in \mathcal{A}_{s}(P;d)$, any rescaling $f_\delta(x) := f(\delta x)$, for $\delta >0$, also belongs to $\mathcal{A}_{s}(P;d)$, and the product $r(Pf)\,.\,r\big( s (-i)^\frak{r}P\f \,\big)$ is invariant. This is due to condition (P4). 

\smallskip

A particularly interesting case is when $P$ is a homogeneous polynomial of degree $\gamma \in \N \cup \{0\}$ in $d$ variables. In this case, the integral conditions in the definition of $\mathcal{A}_{s}(P;d)$ are equivalent to conditions given by the differential operator associated to $P$ applied to  $f$ and $\f$ and evaluated at the origin (provided $f$ and $\f$ are sufficiently smooth). Note that, in principle, we do not require in this case that $|x|^{\gamma}f, |x|^{\gamma}\f \in L^1(\R^d)$, but only the minimal integrability condition $Pf, P\f \in L^1(\R^d)$. The class $\mathcal{A}_{s}(d)$ considered in \eqref{20200410_17:03} corresponds to the case $P = {\bf 1}$.

\vspace{0.2cm}

The question on whether the uncertainty principle holds for the families $\mathcal{A}_{s}(P;d)$, and even the question on whether these families are at least non-empty, may possibly depend on the function $P$; and finding necessary and sufficient conditions seems to be a subtle issue. Before moving into that discussion, let us observe that we can restrict the search to a certain subclass $\mathcal{A}^{*}_{s}(P;d) \subset \mathcal{A}_{s}(P;d)$ of eigenfunctions defined by
\begin{equation}\label{20200508_13:30}
\mathcal{A}^{*}_{s}(P;d) = \left\{
                \begin{array}{ll}
                  f\in L^1(\R^d)\setminus\{{\bf 0}\}\text{ continuous, real-valued and such that } \f= si^{\frak{r}}f;\\
                 Pf \in L^1(\R^d);\\
                 \int_{\R^d} Pf \leq 0\,; \\
                  P f  \text{ is eventually non-negative}.
                \end{array}
              \right\}.
\end{equation}
We also define
\begin{equation}\label{20200508_13:31}
\mathbb{A}^{*}_{s}(P;d) = \inf_{f\in\mathcal{A}^{*}_{s}(P;d) } r(Pf).
\end{equation}
Assuming that the class $\mathcal{A}_{s}(P;d)$ is non-empty, we claim that $\mathcal{A}^*_{s}(P;d)$ is also non-empty and that
\begin{equation}\label{20200510_10:19}
\mathbb{A}_{s}(P;d) = \mathbb{A}^{*}_{s}(P;d).
\end{equation}
To see this, start with any function $f \in \mathcal{A}_{s}(P;d)$. By taking an appropriate rescaling $f_\delta(x) := f(\delta x)$, we may assume that $r(Pf) = r\big( s (-i)^\frak{r}P\f \, \big)$. Observe that $s (-i)^\frak{r}\f \in \mathcal{A}_{s}(P;d)$ and let 
$$w = f + s(-i)^{\frak{r}}\f.$$
Then $\widehat{w} = si^{\frak{r}} w$,  $\int_{\R^d} Pw \leq 0$ and $r(Pw) \leq r(Pf)$. Note that $w$ is not identically zero. In fact, if $w={\bf 0}$, we would have $Pf$ and $s(-i)^{\frak{r}}P\f  = -Pf$ eventually non-negative, which would make $Pf$ eventually zero. By condition (P3) we would have $f = {\bf 0}$, a contradiction. Hence $w \in \mathbb{A}^{*}_{s}(P;d)$ and does a job at least as good as the original $f$.

\smallskip

This is how the eigenfunctions of the Fourier transform (now with all possible eigenvalues) play a role in this discussion. Observe that we may consider directly the {\it eigenfunction extremal problem} described in \eqref{20200508_13:30} - \eqref{20200508_13:31}. In this case, {\it we do not need to assume conditions} (P3) and (P4) for our function $P: \R^d \to \R$, leaving us essentially with the fully generic setup of (P1) and (P2). When we consider the eigenfunction formulation in the results below, the reader should keep in mind the original formulation \eqref{20200510_10:23} - \eqref {20200411_12:04}, and identity \eqref{20200510_10:19}, if applicable.

\smallskip

Note that all of our conditions (P1) -- (P9) are invariant under rotations and reflections. Letting $O(d)$ be the group of linear orthogonal transformations in $\R^d$, if $R \in O(d)$ one can verify that $\mathbb{A}^*_{s}(P;d) = \mathbb{A}^*_{s}(P\circ R;d)$ and $\mathbb{A}_{s}(P;d) = \mathbb{A}_{s}(P\circ R;d)$ by a suitable change of variables.

\smallskip

It is important to emphasize that we {\it do not} identify functions $P$ that are equal almost everywhere with respect to the Lebesgue measure. In fact, even if two functions $P_1$ and $P_2$ are equal a.e., the two problems \eqref{20200508_13:30} - \eqref{20200508_13:31} that they generate may be very different. Consider for example, in dimension $d=1$, $P_1 = {\bf 1}$ and $P_2(x) = 1$ for all $x \in \R \setminus \{a_n\}_{n \in \Z}$, $P_2(a_n) = -1$, where $\{a_n\}_{n \in \N}$ is a given sequence of points with $\lim_{n \to \infty}|a_n| = \infty$. Any function $f \in \mathcal{A}^{*}_{s}(P_2;1)$ will necessarily have zeros at $a_n$ for $n \geq n_0$. In this regard, even problems where $P(x) = 0$ a.e. are non-trivial, and we quickly realize that we are in a vastly uncharted territory.

\smallskip

We first move in the direction of identifying some important situations when these classes are non-empty and providing reasonable upper bounds.

\begin{theorem}[Non-empty classes and upper bounds]\label{Thm_Harmonic_part}
Let $P:\R^d \to \R$ be a function verifying properties {\rm (P1), (P2)} and {\rm(P8)}. Assume that $P = H\,.\, Q$, where $H:\R^d \to \R$ is a homogeneous and harmonic polynomial of degree $\ell \geq 0$, and $Q:\R^d \to \R$ is eventually non-negative. Then $\mathcal{A}^{*}_{s}(P;d)$ is non-empty. If, in addition, $P$ verifies {\rm (P4)}, letting $\frak{r}$ as in \eqref{20200404_03:06} and $\gamma > -d$ as in \eqref{20200401_00:18} we have
\begin{equation*}
 \mathbb{A}^*_{s}(P;d) \leq \sqrt{\frac{\max\{d + \ell + \gamma\, ,\, \ell - \gamma\}}{2\pi}} \ + O(1),
\end{equation*}
with the implied constant being universal; in fact, when $s\,i^{\ell + \frak{r} } = -1$ and $-d <\gamma \leq -\frac{d}{2}$ we have 
\begin{equation*}
 \mathbb{A}^*_{s}(P;d) = 0.
\end{equation*}
\end{theorem}
\noindent{\sc Remark:} Note that in Theorem \ref{Thm_Harmonic_part} we may have $\ell >\gamma$. A simple example would be\footnote{Recall that $\sgn:\R \to \R$ is defined by $\sgn(t) =1$, if $t>0$; $\sgn(0) = 0$; and $\sgn(t) =-1$, if $t<0$.} $P(x) = \sgn (x_1)$, in which $H(x) = x_1$, and $Q(x) = \sgn(x_1)/x_1$ for $x_1\neq 0$ and zero otherwise. We shall not be particularly interested in more explicit quantitative estimates for the upper bounds here.

\smallskip

There is an interesting relationship between the sign uncertainty principles and other classical uncertainty principles. For our purposes, the relevant inequality would be an analogue of \eqref{20200324_00:231}, with $L^1$-norms on the right-hand side. For instance, a basic application of the Hausdorff-Young inequality yields
\begin{equation*}
\|f\|^2_2 \leq \|f\|_1\, \|\f\|_1
\end{equation*}
for any $f \in L^2(\R^d)$, and similar ideas used to prove \eqref{20200324_00:22}, coupled with the Hausdorff-Young inequality, yield
\begin{equation}\label{20200401_21:43}
\|f\|^2_2 \leq 4\pi\,\|x_1f\|_1\, \|x_1\f\|_1
\end{equation}
for any $f \in L^2(\R^d)$ (see, for instance, \cite[Corollary 2.6 and Section 3]{FSsurvey}). By a change of variables given by any $R \in O(d)$ we see that \eqref{20200401_21:43} holds with the function $x_1$ replaced by any linear homogeneous polynomial in $x_1, x_2, \ldots, x_d$. Motivated by such examples we now define a class of admissible functions $P$ that will play an important role in our study. As we shall see, this will be an asset (but not the only one) in establishing sign uncertainty principles.

\enlargethispage{1\baselineskip}

\begin{definition*}[Admissible functions] A function $P:\R^d \to \R$ verifying properties {\rm (P1)} and {\rm (P2)} is said to be admissible if there exists an exponent $q$ with $1 \leq q \leq \infty$ and a positive constant $C(P;d;q)$ such that:
\begin{enumerate}
\item[(i)] For all $f \in L^1(\R^d)$, with $\f= \pm i^{\frak{r}}f$ and $Pf \in L^1(\R^d)$, we have
\begin{equation}\label{20200331_11:16}
    ||f||_q\le C(P;d;q) \, \| P f \|_1.
\end{equation}
\item[(ii)] If $q >1$ we have $P \in L^{q'}_{{\rm loc}}(\R^d)$. If $q =1$ we have $\lim_{r \to 0^+} \|P\|_{L^{\infty}(B_r)} = 0$.\footnote{Throughout the text $1/q + 1/q' = 1$.}
\end{enumerate}
\end{definition*}
The fact that $\f= \pm i^{\frak{r}}f$, together with the Hausdorff-Young inequality, directly implies that $\|f\|_q \leq \|f\|_1$ for all $1 \leq q \leq \infty$. Hence, if \eqref{20200331_11:16} holds for $q=1$, it holds for any exponent $1 \leq q \leq \infty$ with $C(P;d;q) \leq C(P;d;1)$. The finiteness of the sub-level sets is related to the concept of admissibility as our next result shows.

\begin{theorem}[Sufficient conditions for admissibility]\label{Thm_finite_level_sets}
Let $P:\R^d \to \R$ be a function verifying properties {\rm (P1), (P2)} and {\rm (P5)}. Then inequality \eqref{20200331_11:16} holds with $q=1$. In particular, $P$ is admissible with respect to $q = \infty$. If, in addition, $P$ verifies property {\rm (P4)} with degree $\gamma \geq 0$ in \eqref{20200401_00:18}, we can bound the constant $C(P;d;1)$ as:
\begin{itemize}
\item[(i)] If $\gamma = 0$ then 
\begin{equation}\label{20200509_16:53}
C(P;d;1) \leq \big({\rm ess \, inf}|P|\big)^{-1} < \infty.
\end{equation}
\item[(ii)] If $\gamma >0$ then
\begin{equation}\label{20200416_22:56}
C(P;d;1)\leq \left( 1 + \frac{\gamma}{d}\right) \left[ \left( 1 + \frac{d}{\gamma}\right) |A_1|\right]^{\frac{\gamma}{d}}.
\end{equation}
\end{itemize}
\end{theorem}
\noindent{\sc Remark:} Note that in the case $P$ homogeneous of degree $\gamma > 0$, the sub-level set $A_{\lambda}$ has finite measure (for any $\lambda >0$) if and only if 
\begin{equation*}
\int_{\mathbb{S}^{d-1}} |P(\omega)|^{-d/\gamma} \,\d \sigma(\omega) < \infty,
\end{equation*}
where $\sigma$ denotes the surface measure on the unit sphere $\mathbb{S}^{d-1} \subset \R^d$. 

\smallskip

In light of example \eqref{20200401_21:43}, note that Theorem \ref{Thm_finite_level_sets} is not a necessary condition for a function $P$ to be admissible. We are now in position to present a general version of the sign uncertainty principle associated to a function $P$. 

\begin{theorem}[Sign uncertainty]\label{Thm_general_sign_uncertainty}
Let $P:\R^d \to \R$ be a function verifying properties {\rm (P1)} and {\rm (P2)}. Assume that the class $\mathcal{A}^*_{s}(P;d)$ is non-empty and that $P$ is admissible with respect to an exponent $1 \leq q \leq \infty$. Then there exists a positive constant $C^*(P;d;q)$ such that  
\begin{equation}\label{20200513_15:28}
\mathbb{A}^*_{s}(P;d) \geq C^*(P;d;q).
\end{equation}
Moreover, 
\begin{itemize}
\item[(i)] If $P$ verifies properties {\rm (P5), (P7)} and {\rm (P9)}, there exist extremizers for $\mathbb{A}^*_{s}(P;d)$.
\smallskip
\item[(ii)] If $P$ verifies properties {\rm (P4)} and {\rm (P7)}, with degree $\gamma \geq 0$ in \eqref{20200401_00:18} and $K := \|P\|_{L^{\infty}(B_1)}$, 
\begin{equation}\label{20200402_00:50}
C^*(P;d;q)  \geq  \left(\frac{(d + \gamma q')\,\Gamma(d/2)}{2 \,\pi^{\frac{d}{2}} \,(2KC)^{q'}}\right)^{\frac{1}{(d + \gamma q')}},
\end{equation}
where $C = C(P;d;q)$ as in \eqref{20200331_11:16}. If $q = 1$ {\rm(}and hence $\gamma >0${\rm)}, the right-hand side of \eqref{20200402_00:50} should be understood as $(2KC)^{-1/\gamma}$.
\end{itemize}
\end{theorem}
\noindent{\sc Remark:} The constant $C^*(P;d;q)$ in \eqref{20200513_15:28} will be described in the proof. In the homogeneous case (ii) above, under (P5), we can use the fact that $C(P;d;q) \leq C(P;d;1)$ and \eqref{20200509_16:53} - \eqref{20200416_22:56} to get explicit lower bounds in \eqref{20200402_00:50} (that could be then optimized over $q$). In the original case $P = {\bf 1}$ of Theorem A, we can simply choose $q = \infty$ to recover the lower bound $\frac{1}{\sqrt{\pi}} \big(\frac{1}{2}\Gamma\big(\frac{d}{2} + 1\big) \big)^{1/d} > \frac{\sqrt{d}}{\sqrt{2\pi e}}$ as in \cite[Theorem 3]{BCK} and \cite[Theorem 1.4]{CG}. We shall see that, once the non-emptiness and admissibility conditions are in place, the proof of \eqref{20200513_15:28} is rather simple, following the somewhat rigid original scheme of Bourgain, Clozel and Kahane \cite{BCK}. One then realizes that the crux of the matter here is in fact obtaining such conditions, and that is where results like Theorems \ref{Thm_Harmonic_part} and \ref{Thm_finite_level_sets} enter. When $q = \infty$, there is an alternative approach to arrive at the same qualitative conclusion as in \eqref{20200513_15:28} via the operator framework of \cite[Theorem 1]{GOR1}, as communicated to us by F. Gon\c{c}alves. In that statement one could consider $(X, \mu) = (Y, \nu) = (\R^d, |P|\,\d x)$; $p=q=2$; $b=c=1$; and $\mathcal{F} = \{(\sgn(P)f,\, s\cdot  \sgn(P)f) \, ; \, f \in \mathbb{A}^*_{s}(P;d)\}$. The relevant condition that needs to be checked is that $\|\sgn(P)f\|_{L^{\infty}(\R^d, \nu)} \leq  a\, \|\sgn(P)f\|_{L^{1}(\R^d, \mu)}$. This follows from the admissibility condition \eqref{20200331_11:16} with $q = \infty$ (which for instance, under (P5), follows from Theorem \ref{Thm_finite_level_sets}) since $\|\sgn(P)f\|_{L^{\infty}(\R^d, \nu)} \leq \|f\|_{L^{\infty}(\R^d)} \leq  C(P;d;\infty) \|Pf\|_{L^{1}(\R^d)} = C(P;d;\infty)  \|\sgn(P)f\|_{L^{1}(\R^d, \mu)}$. Then, with $r = r(Pf)$, \cite[Theorem 1, Eq. (1.4)]{GOR1} yields $\big\|P \chi_{B_r}\big\|_{1} \geq \big(4\,C(P;d;\infty)\big)^{-1}$, qualitatively as in \eqref{20200511_22:22} below. There are also occasions, as exemplified in \eqref{20200401_21:43}, where the admissibility exponent $q$ is not, in principle, $1$ or $\infty$.

\smallskip

As already mentioned, Theorems \ref{Thm_Harmonic_part} and \ref{Thm_finite_level_sets} can be used to generate a great variety of examples where the hypotheses of Theorem \ref{Thm_general_sign_uncertainty} are verified. A simple example would be $P(x) = |x|^{\gamma}$, for $\gamma \geq0$, while a less straightforward one could be $P:\R^3 \to \R$ given by $P(x) = \big(x_1^2 +x_2^2 - 2 x_3^2\big) \big(x_1^2 +x_2^2 + 2 x_3^2\big)$. The odd functions $P(x) = \sgn(x_1)$ and $P(x) = x_1$ also verify the hypotheses of Theorem \ref{Thm_general_sign_uncertainty} (the latter is admissible directly from \eqref{20200401_21:43}), and these provide two simple versions of sign uncertainty principles associated to the eigenvalues $\pm i$ in all dimensions. In the case $P(x) = \sgn(x)$ in dimension $d=1$, the integral conditions defining the class $\mathcal{A}_{s}(\sgn(x);1)$ can be recast in terms of the sign of the Hilbert transform at the origin. A different sign uncertainty principle for bandlimited functions involving the Hilbert transform appears in \cite[Theorem 4.2]{GOR1}. 

\subsection{Dimension shifts} There will be occasions where the admissibility inequality \eqref{20200331_11:16}, or suitable variants of it, are not, in principle, available (see, for instance, the last remark in Section \ref{Sec_Lap}). We present now a different tool to obtain the sign uncertainty that may be helpful in such circumstances. The intuitive idea is to allow ourselves some movement between different dimensions in order to fall in a favourable situation as in Theorem \ref{Thm_general_sign_uncertainty}. The classical Bochner's relation will be a crucial ingredient in this process and, therefore, radial functions play an important role. In some special situations we are able to go further and establish a surprising identity connecting the sign uncertainty in different dimensions. The reach of the next result will be exemplified in its two companion corollaries. In what follows, for a function $H:\R^d \to \R$ we denote its orbit under the action of the group $O(d)$ by
$$H \circ O(d):= \{ H\circ R:\R^d \to \R\ : \ R \in O(d)\}.$$

\begin{theorem}[Dimension shifts]\label{Dim_Shift} Let $\ell\geq 0$ be an integer and let $\frak{r}(\ell) \in \{0,1\}$ be such that $\frak{r}(\ell) \equiv \ell \,{\rm(mod} \, 2{\rm)}$. Let $P:\R^{d+2\ell} \to \R$ be a function verifying properties {\rm (P1), (P2)} and {\rm (P3)} that is radial. Write $P(x) = P_0(|x|)$. Let $\widetilde{P}:\R^d \to \R$ be a function verifying properties {\rm (P1)} and {\rm (P2)} of the form
\begin{equation}\label{20200514_07:52}
\widetilde{P}(x) = H(x) \, P_0(|x|)\, Q(x),
\end{equation}
where $H:\R^d \to \R$ is a non-zero homogeneous and harmonic polynomial of degree $\ell$ and $Q:\R^d \to \R$ is an even non-negative function, homogeneous of degree $0$. If $\mathcal{A}^*_{s}(P;d+2\ell)$ is non-empty, then $\mathcal{A}^*_{s(-1)^{(\frak{r(\ell)} + \ell)/2}}\big(\widetilde{P};d\big)$ is also non-empty and 
\begin{equation}\label{20200513_16:45}
\mathbb{A}^*_{s}(P;d+2\ell) \geq \mathbb{A}^*_{s(-1)^{(\frak{r(\ell)} + \ell)/2}}\big(\widetilde{P};d\big).
\end{equation}
If, in addition, $P$ verifies property {\rm (P6)}, $Q = {\bf 1}$ and $H \in (x_1x_2\ldots x_{\ell})\circ O(d)$ {\rm(}$0 \leq \ell \leq d${\rm)}, the converse holds: $\mathcal{A}^*_{s}(P;d+2\ell)$ is non-empty if and only if $\mathcal{A}^*_{s(-1)^{(\frak{r(\ell)} + \ell)/2}}\big(\widetilde{P};d\big)$ is non-empty and 
\begin{equation}\label{20200514_10:50}
\mathbb{A}^*_{s}(P;d+2\ell) = \mathbb{A}^*_{s(-1)^{(\frak{r(\ell)} + \ell)/2}}\big(\widetilde{P};d\big).
\end{equation}
\end{theorem}

In general, it is not clear that we can reverse inequality \eqref{20200513_16:45}. One of the main obstacles is to show that the search for the infimum on the right-hand side of \eqref{20200513_16:45} can be reduced to functions $f$ of the form $Hf_0$ with $f_0$ radial (which may simply not be true in general). In the case presented in \eqref{20200514_10:50} we overcome this and other barriers. Our proof also yields the following fact: if there exist extremizers for either side of \eqref{20200514_10:50}, then there exist extremizers for both sides and we have a recipe to explicitly construct one from the other; this is particularly useful to construct explicit extremizers in the situations of Corollary \ref{Sharp_const} below.

\smallskip

We can consider in \eqref{20200514_10:50}, for instance, $P(x) = |x|^{\gamma}$ for $\gamma \geq 0$. In the particular case $P = {\bf 1}$, identity \eqref{20200514_10:50}, together with \eqref{20200327_14:00} and \eqref{20200327_14:01}, yields the following additional $14$ sharp constants (modulo symmetries given by the orthogonal group) in this rough environment of sign uncertainty.
\begin{corollary}[Sharp constants]\label{Sharp_const} Let $\frak{r}(\ell) \in \{0,1\}$ be such that $\frak{r}(\ell) \equiv \ell \,{\rm(mod} \, 2{\rm)}$. Then
\begin{align*}
&\mathbb{A}_{(-1)^{(\frak{r(\ell)} + \ell +2)/2}}\big((x_1\ldots x_{\ell})\circ R\,\,;\, 8 - 2\ell\big) = \sqrt{2},  \ \ \  {\rm for} \ \  0 \leq \ell \leq 2 \  \ {\rm and}\ \  R \in O(8 - 2\ell);\\
&\mathbb{A}_{(-1)^{(\frak{r(\ell)} + \ell)/2}}\big((x_1\ldots x_{\ell})\circ R\,\,;\, 12 - 2\ell\big) = \sqrt{2},   \ \ \  {\rm for} \ \  0 \leq \ell \leq 4 \  \ {\rm and}\ \  R \in O(12 - 2\ell);\\
&\mathbb{A}_{(-1)^{(\frak{r(\ell)} + \ell +2)/2}}\big((x_1\ldots x_{\ell})\circ R\,\,;\, 24 - 2\ell\big) = 2,  \ \ \  {\rm for} \ \  0 \leq \ell \leq 8 \  \ {\rm and}\ \  R \in O(24 - 2\ell).
\end{align*}
\end{corollary}
\noindent{\sc Remark:} {\it A posteriori}, it is worth reflecting on the difficulties of taking a more classical and direct path (e.g. via Poisson-like summation formulas) to approach the sharp constants in Corollary \ref{Sharp_const}. It is also interesting to further investigate the potential connections of this weighted setup and the sharp constants in Corollary \ref{Sharp_const} to other optimization problems in diophantine geometry.

\smallskip

Inequality \eqref{20200513_16:45} is particularly useful in situations where $P$ is singular near the origin (e.g. radially decreasing). In such cases, one can take $Q = |x|^{\ell}\,\sgn(H) / H$ (for $H \neq 0$, and zero otherwise) in \eqref{20200514_07:52} and make $\widetilde{P}$ less singular. Of course, this comes at the expense of lowering the dimension, and there is an intrinsic threshold on how far one can go. For instance, let us come back to the natural power weight $P(x) =|x|^{\gamma},$ where $\gamma > -d$ is a real number. If $\gamma \geq 0$, Theorems \ref{Thm_Harmonic_part}, \ref{Thm_finite_level_sets} and \ref{Thm_general_sign_uncertainty} can be applied and we are in good shape. Note that, in this case, the integral conditions defining the class $\mathcal{A}_{s}(|x|^{\gamma};d)$ can be reformulated in terms of the sign of the fractional Laplacian $(-\Delta)^{\gamma/2}$ of $f$ and $\f$, evaluated at the origin. A related sign uncertainty principle for bandlimited functions and powers of the Laplacian was considered by Gorbachev, Ivanov and Tikhonov in \cite{GIT}. The case $-d < \gamma < 0$ is subtler, and we can bring Theorem \ref{Dim_Shift} into play. In fact, in this situation, we are able to prove or disprove the sign uncertainty principle in a set of ``full density" as the dimension $d$ grows. 
\begin{corollary}[Power weights]\label{Thm_Laplacian}
Let $s \in \{+1, -1\}$ and $\gamma > -d$ be a real number. Let $\varepsilon: \N \to \R$ be defined as: $\varepsilon(d) = 1$ for $d\geq 2$ even, $\varepsilon(1) = \varepsilon(3) = \frac{1}{2}$, and $\varepsilon(d) = \frac{3}{2}$ for $d\geq 5$ odd.
\begin{itemize}
\item[(i)] If  $s =1$ and $\gamma \notin \big(-\frac{d}{2} - \varepsilon(d), -\frac{d}{2} + \varepsilon(d)\big)$ or if $s =-1$ and $\gamma \notin \big(-d, -\frac{d}{2} + \varepsilon(d)\big)$ we have\footnote{Here $\lfloor x \rfloor$ denotes the integer part of $x$, i.e. the largest integer that is smaller than or equal to $x$.}
\begin{equation}\label{20200518_09:54}
c \, \sqrt{\frac{\min\big\{d, |d + 2\lfloor\gamma\rfloor|, |\!-d + 2\lfloor-\gamma\rfloor|\big\}}{2\pi e}} \leq \mathbb{A}_{s}(|x|^{\gamma};d) \leq \sqrt{\frac{\max\{d + \gamma\, ,\, - \gamma\}}{2\pi}} \ + O(1) \,, 
\end{equation}
where $c$ is a positive universal constant. Moreover, if $\gamma \geq 0$, there exists a radial extremizer for $\mathbb{A}_{s}(|x|^{\gamma};d)$.

\smallskip

\item[(ii)] If $s = -1$ and $\gamma \in \big(-d, -\frac{d}{2}\big]$ then
\begin{equation}\label{20200518_09:59}
\mathbb{A}_{-1}(|x|^{\gamma};d) =0.
\end{equation}

\end{itemize}
\end{corollary}

The upper bound in \eqref{20200518_09:54} actually holds for all $\gamma > -d$ and $s = \pm1$. In the proof of this corollary we give a more explicit lower bound in the parameters $d$ and $\gamma$ (that, in particular, recovers the known bounds in the case $\gamma = 0$; see the remark after Theorem \ref{Thm_general_sign_uncertainty}). The uniform lower bound presented in \eqref{20200518_09:54} holds with constant $c = 0.8595\ldots$ if $d=1$; or $d=3$ and $\gamma <0$; and with constant $c=1$ in all other cases. Numerical simulations suggest that the sign uncertainty principle should still hold in the small uncovered neighborhood (of size at most $3$ when $s=1$ and size at most $\frac{3}{2}$ when $s = -1$) around the central point $-\frac{d}{2}$ of the negative range.

\subsection{Notation} Throughout the paper, $B_{\varepsilon}(x)$ denotes the open ball of center $x$ and radius $\varepsilon$ in $\R^d$. If $x=0$ we may simply write $B_{\varepsilon}$. For a measurable set $A \subset \R^d$, we denote its Lebesgue measure by $|A|$ and its characteristic function by $\chi_A$. The function which is identically equal to $0$ (resp. $1$) is denoted by ${\bf  0}$ (resp. ${\bf  1}$). We say that $A \lesssim B$ when $A \leq C\,B$ for a certain constant $C$. We say that $A \simeq B$ when $A \lesssim B$ and $B \lesssim A$.

\section{Non-empty classes and upper bounds: proof of Theorem \ref{Thm_Harmonic_part}}
An important ingredient in this work is the following classical identity.

\begin{lemma}[Bochner's relation]\label{HB_form}
Let $H:\R^d\rightarrow \R$ be a homogeneous, harmonic polynomial of degree $\ell$, and $h:[0,\infty) \to \R$ be a function such that 
\begin{equation*}
\int_{0}^{\infty}|h(r)|^2\,r^{d+2\ell-1}\d r <\infty. 
\end{equation*} 
Let $h_d: \R^d \to \R$ be the radial function on $\R^d$ induced by $h$, that is $h_d(x) := h(|x|)$. Then
\begin{equation*}
\mathcal{F}_d [H\cdot h_d](\xi) = (-i)^{\ell}H(\xi) \cdot\mathcal{F}_{d+2\ell} [h_{d+2\ell}](\xi,0),
\end{equation*}
where $\xi \in \R^d$ and $(\xi, 0) \in \R^d \times \R^{2\ell}$.
\end{lemma}
\begin{proof} See \cite[Chapter III, Theorem 4 and its corollary]{S}.
\end{proof}
We now move to the proof of Theorem \ref{Thm_Harmonic_part}. From \eqref{20200404_03:06} we have
\begin{equation}\label{20220326_17:58}
(-1)^{\frak{r}}H(x)Q(x) = H(-x)Q(-x) = (-1)^{\ell}H(x)Q(-x)
\end{equation}
for all $x \in \R^d$.

\subsection{Non-empty classes} \label{non-empty classes}If $\ell$ and $\frak{r}$ have a different parity, we conclude from \eqref{20220326_17:58} that $P$ must be eventually zero (since $Q$ is assumed to be eventually non-negative). In this case, let $ f\in L^1(\R^d)\setminus\{{\bf 0}\}$ be a continuous and real-valued eigenfunction with $\f= si^{\frak{r}}f$. One plainly sees that either $f$ or $-f$ belongs to $\mathcal{A}^{*}_{s}(P;d)$.

\smallskip

If $\ell$ and $\frak{r}$ have the same parity, we proceed inspired by an example of Bourgain, Clozel and Kahane \cite{BCK}. We consider functions of the form:
\begin{equation}\label{20200516_10:38}
g_0(x)= H(x) \left(e^{-\frac{1}{a_0}\pi |x|^2} + a_0^\frac{d+2\ell}{2}e^{-a_0\pi |x|^2} \right)  \  ;  \ \!h_0(x) = H(x)\, e^{-\pi |x|^2} \  ;  \ f_0(x) = g_0(x) - A_0\, h_0(x),
\end{equation}
and
\begin{align}\label{20200522_22:32}
\begin{split}
g_1(x) =   H(x)& \left(e^{-\frac{1}{a_1}\pi |x|^2}  - a_1^{\frac{d + 2\ell}{2}}\,e^{-a_1\pi |x|^2}\right) \ \ ; \ \ h_1(x) = H(x) \left(e^{-\frac{1}{b_1}\pi |x|^2} - b_1^{\frac{d + 2\ell}{2}}\,e^{-b_1\pi |x|^2}\right); \\
&  \ \ \ \ \ \ \ \ \ \ \ \ \ \ \ \ \ \ \ \ \ \ \  f_1(x) = g_1(x)  - A_1 h_1(x),
\end{split}
\end{align}
with constants $1 < a_0$, $1 < b_1 < a_1$, $A_0$ and $A_1$ arbitrary. Using Lemma \ref{HB_form} we observe that $\widehat{g}_m=(-1)^m(-i)^{\ell}g_m$, $\widehat{h}_m=(-1)^m(-i)^{\ell}h_m$, $\f_m=(-1)^m(-i)^{\ell}f_m$, for $m \in \{0,1\}$. Since $\ell$ and $\frak{r}$ have the same parity, when $(-1)^m = s\,i^{\ell+\frak{r}}$ these are eigenfunctions with the desired eigenvalue $si^{\frak{r}}$. Note that $Pg_m$, $Ph_m$ and $Pf_m$ are eventually non-negative (and integrable due to property (P8)). If either $\int_{\R^d} Pg_m \leq0$ or $\int_{\R^d} Ph_m \leq 0$, that function will belong to the class $\mathcal{A}^{*}_{s}(P;d)$. If both of these integrals are positive, we adjust the constant $A_m$ to make $\int_{\R^d} Pf_m \leq0$ and hence $f_m \in \mathcal{A}^{*}_{s}(P;d)$. This shows that $\mathcal{A}^{*}_{s}(P;d)$ is non-empty.

\subsection{Homogeneous case} Assume now that $P$ verifies (P4). As discussed in \S \ref{non-empty classes}, in this situation we must have $\ell$ and $\frak{r}$ with the same parity.

\subsubsection{Case $s i^{\ell+\frak{r}} = 1$} In this case we work with the function $f_0$ in \eqref{20200516_10:38} and let 
$$A_0 = a_0^{\frac{d+\ell+\gamma}{2}} + a_0^{\frac{\ell - \gamma}{2}}.$$
From the homogeneity of $P$ and $H$ one can check that this choice of $A_0$ yields $\int_{\R^d} Pf_0 = 0$. Note from \eqref{20200516_10:38} that 
$$Pf_0 \geq PH \, e^{-\pi |x|^2} \left( e^{(1 -\frac{1}{a_0})\pi |x|^2} - A_0\right)$$
for $x \neq 0$ (recall that $PH = H^2Q$ being homogeneous and eventually non-negative is actually non-negative outside the origin). This plainly implies that 
$$r(Pf_0) \leq \sqrt{\frac{a_0 \log A_0}{\pi (a_0 - 1)}}\ .$$
Let $\rho := \max\{d + \ell + \gamma\, ,\, \ell - \gamma\} \geq \frac{d}{2}\,,$ and let $a_0 = 1 + \alpha$, with $0 < \alpha \leq \sqrt{2}$ to be chosen. Using that $1 \leq A_0 \leq 2 a_0^{\rho/2}$ and that 
$$\frac{a_0 \log a_0}{(a_0 -1)} = 1 + O(\alpha) \ \ ; \ \ \frac{a_0}{(a_0 -1)}  = \frac{1}{\alpha} + 1,$$
we find
\begin{equation}\label{20200516_10:56}
r(Pf_0) \leq \sqrt{\frac{\, a_0 \big((\rho/2) \log a_0 + \log 2\big)}{\pi (a_0 - 1)}} = \sqrt{\frac{\rho}{2\pi} \big(1 + O(\alpha)\big) + O\left(\frac{1}{\alpha}\right)}.
\end{equation}
We now choose $\alpha = \frac{1}{\sqrt{\rho}}$. Then \eqref{20200516_10:56} reads
\begin{equation*}
r(Pf_0) \leq \sqrt{\frac{\rho}{2\pi} + O(\sqrt{\rho})} = \sqrt{\frac{\rho}{2\pi}} + O(1),
\end{equation*}
as we wanted.

\subsubsection{Case $s i^{\ell+\frak{r}} = -1$ and  $-d < \gamma \leq -\frac{d}{2}$} In this situation we consider the function $g_1$ in \eqref{20200522_22:32}. Using the homogeneity we note that 
\begin{equation*}
\int_{\R^d} P(x)\,g_1(x)\,\d x = \left(\int_{\R^d} P(x)\,H(x) \,e^{-\pi |x|^2}\d x \right) \left(a_1^{\frac{d+\ell+\gamma}{2}} - a_1^{\frac{\ell - \gamma}{2}}\right) \leq 0
\end{equation*}
for $a_1 >1$. From \eqref{20200522_22:32} we plainly see that 
\begin{equation*}
r(Pg_1) \leq \sqrt{\frac{(d + 2\ell)\log a_1}{2\pi \big(a_1 - \frac{1}{a_1}\big)}} \ \to \ 0
\end{equation*}
as $a_1 \to \infty$. Hence, in this case, $ \mathbb{A}^*_{s}(P;d) = 0$.

\subsubsection{Case $s i^{\ell+\frak{r}} = -1$ and  $-\frac{d}{2} < \gamma$} We now consider $f_1$ in \eqref{20200522_22:32} with the choice
$$A_1 = \frac{a_1^{\frac{d+\ell+\gamma}{2}} - a_1^{\frac{\ell - \gamma}{2}}}{b_1^{\frac{d+\ell+\gamma}{2}} - b_1^{\frac{\ell - \gamma}{2}}}.$$
Observe that $\int_{\R^d} Pf_1 = 0$. We consider $a_1 = 1 + 2\alpha$ and $b_1 = 1 + \alpha$ with $0 < \alpha \leq \sqrt{2}$ to be chosen. Using the expansion 
\begin{equation*}
\frac{a_1 \log a_1}{(a_1^2 -1)} = \frac{1}{2} + O(\alpha),
\end{equation*}
we note that the inequality 
\begin{equation}\label{20200516_11:26}
a_1^{\frac{d + 2\ell}{2}}\,e^{-a_1\pi |x|^2} \leq \frac{1}{2} \, e^{-\frac{1}{a_1}\pi |x|^2}
\end{equation}
holds for all $|x| \geq r_1$, where
\begin{equation}\label{20200516_11:32}
r_1 = \sqrt{\frac{(d + 2\ell)}{2\pi} \left(\frac{1}{2} + O(\alpha)\right) + O\left(\frac{1}{\alpha}\right)}.
\end{equation}
Assuming that \eqref{20200516_11:26} holds, we have that 
\begin{equation*} 
Pf_1 \geq PH \, e^{-\frac{1}{b_1}\pi |x|^2} \left( \frac{1}{2}\,e^{(\frac{1}{b_1} -\frac{1}{a_1})\pi |x|^2} - A_1\right) \ \ \ (x \neq 0).
\end{equation*}
This tells us that $r(Pf_1) \leq \max\{r_1, r_2\}$, with $r_1$ as in \eqref{20200516_11:32} and 
\begin{equation}\label{20200516_12:08}
r_2:= \sqrt{\frac{\log 2 A_1}{\pi} \frac{a_1 b_1}{(a_1 - b_1)}}.
\end{equation}
As before, let $\rho := \max\{d + \ell + \gamma\, ,\, \ell - \gamma\}  = d + \ell + \gamma$ in this case. Observe that we can write $A_1$ as 
\begin{equation*}
A_1 =\left( \frac{a_1}{b_1}\right)^{\rho/2} \left( \frac{1 - a_1^{-\frac{(d + 2\gamma)}{2}}}{1 - b_1^{-\frac{(d + 2\gamma)}{2}}}\right).
\end{equation*}
Since $1 < b_1 < a_1$, one can verify that the function 
$$t \mapsto \left( \frac{1 - a_1^{-t}}{1 - b_1^{-t}}\right)$$
is non-increasing for $t > 0$, with the limit being $\log a_1 / \log b_1$ as $t \to 0^+$. Hence
\begin{equation}\label{20200516_12:07}
1 \leq A_1 \leq \left( \frac{a_1}{b_1}\right)^{\rho/2} \left(\frac{\log a_1}{\log b_1}\right).
\end{equation}
Now we plug in the upper bound \eqref{20200516_12:07} in \eqref{20200516_12:08} and use the expansions 
\begin{equation*}
\frac{a_1 b_1(\log a_1 - \log b_1)}{(a_1 - b_1)} = 1 + O(\alpha) \ \  ;  \ \ \frac{a_1 b_1(\log (\log a_1/ \log b_1))}{(a_1 - b_1)} =  O\left(\frac{1}{\alpha}\right) \ \ ; \ \ \frac{a_1 b_1}{(a_1 - b_1)} = O\left(\frac{1}{\alpha}\right)
\end{equation*}
to find that 
\begin{equation*}
r_2 \leq \sqrt{\frac{\rho}{2\pi} \big(1 + O(\alpha)\big) + O\left(\frac{1}{\alpha}\right)}\ .
\end{equation*}
This is the same as \eqref{20200516_10:56}. We have seen that the choice $\alpha = \frac{1}{\sqrt{\rho}}$ leads to $r_2 \leq  \sqrt{\frac{\rho}{2\pi}} + O(1)$. Since $(d + 2\ell)/2 \leq \rho$, we also have $r_1 \leq  \sqrt{\frac{\rho}{2\pi}} + O(1)$. This concludes the proof of Theorem \ref{Thm_Harmonic_part}.

\subsection{An additional reduction} We briefly present a related result that may be helpful in some situations. This is inspired in similar reductions in \cite{BCK, CG}.

\begin{proposition}\label{Prop_8_integral_zero}
Let $P:\R^d \to \R$ be a function verifying properties {\rm (P1), (P2), (P3)} and {\rm(P4)}. Assume that $P = H\,.\, Q$, where $H:\R^d \to \R$ is a homogeneous and harmonic polynomial of degree $\ell \geq 0$, and $Q:\R^d \to \R$ is a non-negative function. Let $\frak{r}$ as in \eqref{20200404_03:06}. If $s\,i^{\ell + \frak{r} } = 1$, or if $s\,i^{\ell + \frak{r} } = -1$ and $\gamma \geq \ell$, we can reduce the search in \eqref{20200508_13:30} - \eqref{20200508_13:31} to functions verifying $\int_{\R^d} Pf = 0$. 
\end{proposition}
\begin{proof} Assume that $f \in \mathcal{A}^{*}_{s}(P;d)$ is such that $\int_{\R^d} Pf <0$. Note, in particular, that we cannot have $P(x) = 0$ a.e. in this situation. Let us show how we can adjust the function $f$.
\subsubsection*{Case 1: $s i^{\ell + \frak{r} } = 1$}  Let $\varphi(x) = H(x)\,e^{-\pi |x|^2}$. By Lemma \ref{HB_form} we have $\widehat{\varphi} = (-i)^{\ell} \varphi = si^{\frak{r}} \varphi $. Then $\int_{\R^d} P \varphi > 0$ and we may consider 
\begin{equation*}
g(x) = f(x) - \frac{ \int_{\R^d} Pf}{\int_{\R^d} P \varphi} \, \varphi(x).
\end{equation*}
One can verify that $g \neq {\bf 0}$ (otherwise $Pf$ is eventually zero and from (P3) we get a contradiction), $\widehat{g} = si^{\frak{r}} g$, $r(Pg) \leq r(Pf)$ and $ \int_{\R^d} Pg = 0$. 

\subsubsection*{Case 2: $s i^{\ell + \frak{r} } = -1$ and $\gamma \geq \ell$} Let $t>0$ be a parameter to be chosen later, and define
 \begin{equation*}
 \psi_t(x):= H(x) \left(e^{-t\pi|x|^2} - 2^{-\frac{(\gamma -\ell)}{2}}\,e^{-2t\pi |x|^2}\right).
 \end{equation*}
 Then, by Lemma \ref{HB_form},  
 \begin{equation*}
 \widehat{\psi_t}(x) = (-i)^{\ell} H(x) \left(t^{-\frac{(d+2\ell)}{2}}e^{-\frac{\pi|x|^2}{t}} - 2^{-\frac{(\gamma -\ell)}{2}}\,(2t)^{-\frac{(d+2\ell)}{2}} e^{-\frac{\pi |x|^2}{2t}} \right).
\end{equation*}
Observe that $\psi_t$ is a Schwartz function that satisfies $P \psi_t \geq 0$ and $ \int_{\R^d} P \psi_t  > 0$. Using the homogeneity, a change of variables shows that $ \int_{\R^d} P \widehat{\psi_t}  = 0$. Observe also that 
\begin{equation*}
i^{\ell} P(x) \,\widehat{\psi_t}(x)<0\ \ {\rm for }\ \ |x|>\sqrt{\dfrac{(d+\ell +\gamma)\,t\log 2}{\pi}}.
\end{equation*}
We choose $t>0$ such that 
$$r(Pf) = \sqrt{\dfrac{(d+\ell +\gamma)\,t\log 2}{\pi}}$$
and consider 
$$g(x) = f(x) - \frac{ \int_{\R^d} Pf}{\int_{\R^d} P \psi_t} \big(\psi_t(x)-i^{\ell}\,\widehat{\psi_t}(x)\big).$$
One can verify that $g \neq {\bf 0}$ (otherwise $Pf$ is eventually zero and from (P3) we get a contradiction), $\widehat{g} = si^{\frak{r}} g$, $r(Pg) \leq r(Pf)$ and $ \int_{\R^d} Pg = 0$. 
 \end{proof}

\section{Sufficient conditions for admissibility: proof of Theorem \ref{Thm_finite_level_sets}}
For $E \subset \R^d$, recall that $|E|$ denotes its Lebesgue measure, and we let $E^c = \R^d \setminus E$. The following classical result will be useful.
\begin{lemma}[Amrein-Berthier \cite{AB}]\label{AB_unc}
Let $E, F \subset \R^d$ be sets of finite measure. Then there exists a constant $C= C(E,F;d) >0$ such that for all $g \in L^2(\R^d)$ we have
\begin{equation}\label{20200401_18:10}
\int_{\R^d} |g(x)|^2\,\d x \leq C \left( \int_{E^c} |g(x)|^2\,\d x + \int_{F^c} |\widehat{g}(x)|^2\,\d x\right). 
\end{equation}
\end{lemma}
\noindent{\sc Remark:} Later works of Nazarov \cite{N} and Jaming \cite{J} show that \eqref{20200401_18:10} holds with 
$$C(E,F;d) \leq c \,e^{c|E||F|},$$
for some $c = c(d)$.
\subsection{Proof of Theorem \ref{Thm_finite_level_sets}: general case} \label{gen_case_sub_level} Let $f \in L^1(\R^d)$ with $\f= \pm i^{\frak{r}}f$, and let $A = A_{\lambda} = \{x \in \R^d \,:\, |P(x)|\leq \lambda\}$ be of finite Lebesgue measure. For a set $E \subset \R^d$, let $f_E := f \cdot \chi_E$, where $\chi_{E}$ is the characteristic function of $E$. By the triangle inequality and the Cauchy-Schwarz inequality we have
\begin{align}\label{20200401_22:59}
\|f\|_1\leq \big\|f_{A}\big\|_1 + \big\|f_{A^c}\big\|_1 \leq |A|^{1/2}\, \big\|f_{A}\big\|_2 + \big\|f_{A^c}\big\|_1.
\end{align}
In the terminology of Lemma \ref{AB_unc}, let $E=F=A$ and let $C = C(A,A;d)$ in \eqref{20200401_18:10}. Letting $g = f_{A}$ in \eqref{20200401_18:10} we plainly get
\begin{equation*}
\int_{A} |f(x)|^2\,\d x \leq C \int_{A^c} \big|\widehat{f_{A}}(x)\big|^2\,\d x = C \left(   \int_{A} |f(x)|^2\,\d x - \int_{A} \big|\widehat{f_{A}}(x)\big|^2\,\d x\right),
\end{equation*}
and then 
\begin{equation}\label{20200401_22:58}
\int_{A} \big|\widehat{f_A}(x)\big|^2\,\d x \leq \frac{(C -1)}{C} \int_{A} |f(x)|^2\,\d x.
\end{equation}
The fact that $f$ is an eigenfunction yields $f =  \pm (-i)^{\frak{r}} \f = \pm (-i)^{\frak{r}} \big( \widehat{f_A} + \widehat{f_{A^c}}\big)
$ and hence
\begin{align*}
f_A =  \pm (-i)^{\frak{r}} \left( \big(\widehat{f_A}\big)_A + \big(\widehat{f_{A^c}}\big)_A\right).
\end{align*}
A basic triangle inequality then yields
\begin{equation}\label{20200401_22:56}
\|f_A \|_2 \leq \left\| \big(\widehat{f_A}\big)_A\right\|_2 + \left\| \big(\widehat{f_{A^c}}\big)_A\right\|_2.
\end{equation}
We bound the last term in \eqref{20200401_22:56} as follows:
\begin{equation}\label{20200401_22:57}
 \left\| \big(\widehat{f_{A^c}}\big)_A\right\|_2 \leq  \left\| \big(\widehat{f_{A^c}}\big)\right\|_\infty \, |A|^{1/2} \leq \|f_{A^c}\|_1 \  |A|^{1/2}.
\end{equation}
From \eqref{20200401_22:58}, \eqref{20200401_22:56} and \eqref{20200401_22:57} we get
\begin{align*}
\|f_A \|_2 \leq \left( \frac{C-1}{C}\right)^{1/2} \|f_A \|_2 + |A|^{1/2} \, \|f_{A^c}\|_1,
\end{align*}
which implies that
\begin{equation}\label{20200401_23:09}
\|f_A \|_2 \leq \frac{|A|^{1/2}}{ \left(1 - \left( \frac{C-1}{C}\right)^{1/2}\right)} \ \|f_{A^c}\|_1.
\end{equation}
Finally, plugging \eqref{20200401_23:09} into \eqref{20200401_22:59} yields
\begin{align*}
\|f\|_1 \leq \left(1 + \frac{|A|}{ \left(1 - \left( \frac{C-1}{C}\right)^{1/2}\right)}\right)  \|f_{A^c}\|_1 \leq \left(1 + \frac{|A|}{ \left(1 - \left( \frac{C-1}{C}\right)^{1/2}\right)}\right) \lambda^{-1}\, \|Pf\|_1,
\end{align*}
as we wanted.

\subsection{Homogeneous case} If $P$ is homogeneous of degree $\gamma = 0$, inequality \eqref{20200509_16:53} is clear. In this case, $|A_{\lambda}|$ is either $0$ or $\infty$, and therefore (P5) implies that ess$\inf |P|>0$. Assume then that $P$ is homogeneous of degree $\gamma > 0$. In this case, $A_{\lambda} = \lambda^{1/\gamma}A_1$, and hence $|A_{\lambda}| = \lambda^{d/\gamma} |A_1|$. Let us write again $A = A_{\lambda}$ for some $\lambda >0$ to be properly chosen later, with the condition that $|A_{\lambda}| < 1$. By the Hausdorff-Young and Cauchy-Schwarz inequalities we have
\begin{align*}
\big\| \widehat{f_{A}}\big\|_{\infty} \leq \big\|f_{A}\big\|_1 \leq  \big\|f_{A}\big\|_2\,  |A|^{1/2},
\end{align*}
from which we obtain
\begin{equation}\label{20200511_17:32}
\int_{A} \big|\widehat{f_{A}}(x)\big|^2\,\d x \leq |A|^2 \int_{A} |f(x)|^2\,\d x.
\end{equation}
We let estimate \eqref{20200511_17:32} replace \eqref{20200401_22:58} in the proof of the general case in \S \ref{gen_case_sub_level}. If we repeat all the other steps we get
\begin{align*}
\|f\|_1 \leq \left(1 + \frac{|A|}{ \big(1 -|A|\big)}\right)  \big\|f_{A^c}\big\|_1 = \frac{1}{1 - |A|} \,\big\|f_{A^c}\big\|_1 \leq \frac{1}{(1 - \lambda^{d/\gamma}|A_1|)} \,\lambda^{-1} \,\|Pf\|_1.
\end{align*}
We are now free to choose $\lambda >0$ in order to minimize the function 
$$\varphi(t) = \frac{1}{t\,(1 - t^{d/\gamma}|A_1|)},$$
subject to the condition $|A_{\lambda}| = \lambda^{d/\gamma} |A_1| < 1$. The minimum occurs when 
$$\lambda^{d/\gamma} |A_1| = \frac{\frac{\gamma}{d}}{1 + \frac{\gamma}{d}},$$
which gives
$$\varphi(\lambda) = \left( 1 + \frac{\gamma}{d}\right) \left[ \left( 1 + \frac{d}{\gamma}\right) |A_1|\right]^{\frac{\gamma}{d}}.$$

\section{Sign uncertainty: proof of Theorem \ref{Thm_general_sign_uncertainty}}
Throughout this proof we let $\omega_{d-1} = 2 \,\pi^{d/2} \,\Gamma(d/2)^{-1}$ be the surface area of the unit sphere $\mathbb{S}^{d-1} \subset \R^d$. For $y \in \R$ we denote $y_+ =\max\{y,0\}$ and $y_- =  \max\{-y,0\}$. 

\subsection{Lower bound} Let $f \in \mathcal{A}^*_{s}(P;d)$ and let $r$ be arbitrary with $r > r(Pf)$. Then
\begin{align*}
\begin{split}
\int_{\R^d} P(x) \,f(x)\,\d x=  \int_{\R^d} [P(x)f(x)]_+\,\d x - \int_{\R^d} [P(x)f(x)]_-\,\d x \ \leq \ 0.
\end{split}
\end{align*}
Let $B_r = \{x \in \R^d\,:\, |x| < r\}$. By definition, $[Pf]_-$ is supported on $\overline{B_r}$ and $[Pf]_- \leq |Pf|$. Since $P$ is admissible with respect to an exponent $1 \leq q \leq \infty$, by H\"{o}lder's inequality we have
\begin{align}\label{20200511_15:24}
\|Pf\|_1  & =  \int_{\R^d} [Pf]_+  + \int_{\R^d} [Pf]_- \ \leq  \ 2 \int_{\R^d} [Pf]_- \ \leq \ 2 \int_{B_r} |Pf| \  \leq \ 2 \, \big\|P \chi_{B_r}\big\|_{q'}\,  \|f\|_q.
\end{align}
From \eqref{20200331_11:16} and \eqref{20200511_15:24} we get
\begin{align*}
\|f\|_q \leq C(P;d;q) \, \| P f \|_1 \leq 2 \,C(P;d;q) \,\big\|P \chi_{B_r}\big\|_{q'} \,\|f\|_q\,,
\end{align*}
and we conclude that 
\begin{equation}\label{20200511_22:22}
\big\|P \chi_{B_r}\big\|_{q'} \geq \frac{1}{2\,C(P;d;q)}.
\end{equation}
From \eqref{20200511_22:22} we deduce that $r$ is bounded below by a constant, since, by assumption (ii) in the definition of admissible function, we have $\lim_{r \to 0^+} \big\|P \chi_{B_r}\big\|_{q'} = 0$.

\subsection{Homogeneous case} In this case observe that $|P(x)| \leq K |x|^{\gamma}$ and we can directly compute
\begin{align}\label{20200511_22:20}
\big\|P \chi_{B_r}\big\|_{q'} \leq K \left(\int_{B_r} |x|^{\gamma q'}\,\d x\right)^{1/q'}  = K 
\left( \frac{\omega_{d-1}\,r^{d+\gamma q'} }{d+\gamma q'}\right)^{1/q'} 
\end{align}
if $q > 1$. If $q = 1$ (and hence $\gamma >0$ from the admissibility hypotheses) we simply have 
\begin{align}\label{20200511_22:21}
\big\|P \chi_{B_r}\big\|_{\infty} \leq K r^{\gamma}.
\end{align}
Plugging \eqref{20200511_22:20} - \eqref{20200511_22:21} into \eqref{20200511_22:22} yields, for $q >1$, 
\begin{equation}\label{20200416_22:18}
r  \geq \left(\frac{(d + \gamma q')\,\Gamma(d/2)}{2 \,\pi^{\frac{d}{2}} \,(2\,K \, C(P;d;q))^{q'}}\right)^{\frac{1}{(d + \gamma q')}},
\end{equation}
If $q = 1$, the right-hand side of \eqref{20200416_22:18} becomes $(2\,K\, C(P;d;1))^{-1/\gamma}$.

\subsection{Existence of extremizers} The argument to establish the existence of extremizers in certain Fourier optimization problems generally involves showing that a suitable weak limit is a viable candidate. Examples of such methods can be found in \cite{CMS, CG, GOS}.

\smallskip

Let $\{f_n\} \subset \mathcal{A}^*_{s}(P;d)$ be an extremizing sequence. This implies that $r(Pf_n) \to \mathbb{A}^*:= \mathbb{A}^*_{s}(P;d)$, and we may assume that $r(Pf_n)$ is non-increasing. We normalize the sequence so that $\|f_n\|_2 =1$. From the reflexivity of $L^2(\R^d)$, passing to a subsequence if necessary, we may assume that $f_n \rightharpoonup f$ weakly, for some $f \in L^2(\R^d)$. By Plancherel's theorem, note that $\widehat{f_n} \rightharpoonup \f$ and therefore $\widehat{f} = si^{\frak{r}} f$. We now prove that $f$ is equal a.e. to our desired extremizer.

\smallskip

Let $r_1 = r(Pf_1)$. Then $r_1\geq r(Pf_n) \geq \mathbb{A}^*$ for all $n \in \N$. Since we are assuming property (P5), Theorem \ref{Thm_finite_level_sets} tells us that \eqref{20200331_11:16} holds with $q=1$, and hence also with $q =2$. Estimate \eqref{20200511_15:24} also holds for $q=1$ and $q=2$, and under conditions (P1) and (P7) we then have
\begin{align}\label{20200406_15:54}
\|f_n\|_1 \simeq \|f_n\|_2 \simeq \|Pf_n\|_1,
\end{align}
with implied constants only depending on $d, P$ and $r_1$. By Mazur's lemma \cite[Corollary 3.8 and Exercise 3.4]{B}, we can find $g_n$ a finite convex combination of $\{f_n, f_{n+1}, \ldots\}$ such that $g_n \to f$ strongly in $L^2(\R^d)$. Passing to a subsequence, if necessary, we may also assume that $g_n \to f$ almost everywhere. Observe that $g_n$ is not identically zero (since each $Pf_k$ must be strictly positive somewhere in $\{x \in \R^d \, : \, |x| > r_1\}$ due to condition (P3) which is implied by (P5)) and hence $g_n \in \mathcal{A}^*_{s}(P;d)$ with \begin{equation}\label{20200406_18:33}
r_1\geq r(Pf_n) \geq r(Pg_n) \geq \mathbb{A}^*
\end{equation}
for all $n \in \N$. By the triangle inequality, we have $\|g_n\|_2 \leq 1$. The norm equivalences as in \eqref{20200406_15:54} continue to hold for $g_n$. In particular, by Fatou's lemma,
\begin{equation}\label{20200406_18:31}
\|f\|_1 \leq \liminf_{n \to \infty} \|g_n\|_1 \lesssim \|g_n\|_2 \leq 1 \ \ \ \ {\rm and}\  \ \ \ \|Pf\|_1 \leq \liminf_{n \to \infty} \|Pg_n\|_1 \lesssim \|g_n\|_2 \leq 1.
\end{equation}
By the Hausdorff-Young inequality, $\|g_n\|_{\infty} \leq \|g_n\|_1 \lesssim |\|g_n\|_2 \leq 1$, and we may then use dominated convergence to get
\begin{equation}\label{20200406_16:55}
\int_{B_{r_1}} Pf = \lim_{n \to \infty} \int_{B_{r_1}} Pg_n.
\end{equation}
Fatou's lemma again gives us
\begin{equation}\label{20200406_16:56}
\int_{B_{r_1}^c} Pf \leq \liminf_{n \to \infty} \int_{B_{r_1}^c} Pg_n,
\end{equation}
and if we add up \eqref{20200406_16:55} and \eqref{20200406_16:56} we get
\begin{equation*}
\int_{\R^d} Pf  \leq \liminf_{n \to \infty} \int_{\R^d} Pg_n \leq 0,
\end{equation*}
since $g_n \in \mathcal{A}^*_{s}(P;d)$. By \eqref{20200406_18:31}, since $f$ is an integrable eigenfunction, it is equal a.e. to a continuous function, and we make this identification. Once we establish that $f \neq {\bf 0}$, we will have that $f \in \mathcal{A}^*_{s}(P;d)$. In fact, assume that $g_n(x) \to f(x)$ for all $x \in E$, where $\big|\R^d \setminus E\big| = 0$. Then, if $x \in E \cap \overline{B}_{\mathbb{A^*}}^c$, from \eqref{20200406_18:33} we get $P(x)f(x) \geq 0$. Now consider $x \in E^c \cap \overline{B}_{\mathbb{A^*}}^c$ such that $P(x) \neq 0$. From the sign density property (P9) we can find a sequence $x_j \to x$ with $x_j \in E \cap \overline{B}_{\mathbb{A^*}}^c$ and $P(x_j)P(x)>0$. Since $P(x_j)f(x_j) \geq 0$, we have $P(x)f(x_j) \geq 0$ and, by the continuity of $f$, we arrive at $P(x)f(x)\geq 0$. The conclusion is that $P(x)f(x) \geq 0$ for all $|x| >\mathbb{A}^*$. Hence $r(Pf) = \mathbb{A}^*$ and $f$ will be our desired extremizer. 

\smallskip

It remains to show that $f \neq {\bf 0}$. Under (P5), let $A = A_{\lambda} = \{x \in \R^d \,:\, |P(x)|\leq \lambda\}$ be of finite measure. From Lemma \ref{AB_unc} (with $E = F = A \cup B_{r_1}$), the Hausdorff-Young inequality, and \eqref{20200406_15:54} we get 
\begin{align}\label{20200406_18:40}
1 = \|f_n\|_2^2 & \lesssim \int_{B_{r_1}^c \cap A^c} |f_n(x)|^2 \, \d x \leq \|f_n\|_{\infty} \int_{B_{r_1}^c \cap A^c} |f_n(x)|\,\d x\nonumber\\
& \lesssim   \int_{B_{r_1}^c \cap A^c} P(x) \,f_n(x) \,\d x\\
& \leq \int_{B_{r_1}^c} P(x) \,f_n(x) \,\d x. \nonumber
\end{align}
Since $\int_{\R^d} Pf_n \leq 0$ and $Pf_n$ is non-negative in $\overline{B}_{r_1}^c $, estimate \eqref{20200406_18:40} tells us that there is a positive constant $C$ depending only on $d,P$ and $r_1$ such that 
\begin{align*}
\int_{B_{r_1}} P f_n \leq -C.
\end{align*}
The weak convergence directly implies that (note properties (P1) and (P7))
\begin{align*}
\int_{B_{r_1}} P f \leq -C.
\end{align*}
In particular, this shows that  $f \neq {\bf 0}$ and the proof is concluded.

\section{Dimension shifts: proof of Theorem \ref{Dim_Shift}}
\subsection{Dropping the dimension} \label{Dim_drop}Let us first prove inequality \eqref{20200513_16:45} in the generic case. We are assuming that $\mathcal{A}^{*}_{s}(P;d + 2\ell)$ is non-empty. We first observe that the search can be further restricted to radial functions. For this, let $SO(d + 2\ell)$ be the group of rotations in $\R^{d + 2 \ell}$ (linear orthogonal transformations of determinant $1$) with its Haar measure $\mu$, normalized so that $\mu(SO(d + 2\ell))= 1$. For $ f \in \mathcal{A}^{*}_{s}(P;d + 2\ell)$ we define
\begin{equation}\label{20200415_23:02}
f^{{\rm rad}}(x): = \int_{SO(d + 2 \ell)} f(Rx)\,\d \mu(R).
\end{equation}
One can readily check that $f^{{\rm rad}}$ is continuous, that $f^{{\rm rad}}, Pf^{{\rm rad}} \in L^1(\R^{d+2\ell})$, $\int_{\R^{d+2\ell}} P f^{{\rm rad}}\leq 0$, $\widehat{f^{{\rm rad}}} = s f^{{\rm rad}}$, and that $r(Pf^{{\rm rad}}) \leq r(Pf)$. To see that $f^{{\rm rad}} \neq {\bf 0}$ we argue as follows. Let $r= r(Pf)$. From condition (P3), there exists a certain $x_0 \in \R^{d+2\ell}$, with $|x_0| > r$ such that $P(x_0)f(x_0) > 0$. As $P$ is radial, we have $P(x_0)f(Rx_0) \geq 0$ for all $R \in SO(d + 2 \ell)$, with strict inequality if $R$ is in a suitable neighborhood of the identity, since $f$ is continuous. Then $P(x_0)f^{{\rm rad}}(x_0) >0$. The conclusion is that in fact  $f^{{\rm rad}} \in \mathcal{A}^{*}_{s}(P; d + 2\ell)$ (and does a job at least as good as the original $f$).

\smallskip

Now let us start with $f \in \mathcal{A}^{*}_{s}(P;d + 2\ell)$ radial. Write $f(x) = f_0(|x|)$ for some continuous $f_0:[0,\infty) \to \R$. The conditions $f , Pf \in L^1(\R^{d + 2\ell})$ can be rewritten as
\begin{equation}\label{20200409_15:08}
\int_0^{\infty} |f_0(r)|\,r^{d+2\ell -1}\,\d r<\infty     \ \ {\rm and} \ \ \int_0^{\infty} |f_0(r)|\,|P_0(r)|\, r^{d+2\ell -1}\,\d r<\infty.
\end{equation}
Define $f^{\flat}:\R^d \to \R$ by 
\begin{equation}\label{20200409_15:34}
f^{\flat}(x) := H(x) \, f_0(|x|).
\end{equation}
Then $|f^{\flat}(x)|\leq C |x|^{\ell}|f_0(|x|)|$ and \eqref{20200409_15:08} gives us that $|x|^{\ell} \, f^{\flat}, \, |x|^{\ell}|P_0(|x|)| \, f^{\flat} \in L^1(\R^d)$. This plainly implies that $f^{\flat} , \widetilde{P}f^{\flat} \in L^1(\R^d)$. The latter is obvious if $P_0 = 0$ a.e. in $[0,\infty)$ and, if not, observe that property (P1) for $\widetilde{P}$ implies that $HQ \in L^1(\mathbb{S}^{d-1})$. In this second case we also have
\begin{align*}
\int_{\R^d} \widetilde{P}(x) \,f^{\flat}(x)\,\d x &= \int_{\R^d} H(x)^2 Q(x) P_0(|x|)f_0(|x|)\,\d x \\
& = \left(\int_{\mathbb{S}^{d-1}} H(\omega)^2 Q(\omega)\,\d \sigma(\omega) \right)  \int_0^{\infty} P_0(r)\,f_0(r)\, r^{d+2\ell -1}\,\d r \ \leq \ 0,
\end{align*}
since the quantity in parentheses is non-negative (and finite) and $\int_{\R^{d+2\ell}} Pf \leq 0$. 

\smallskip

From Bochner's relation (Lemma \ref{HB_form}) and the fact that $\widehat{f} =  sf$ we have 
\begin{equation*}
\widehat{f^{\flat}}(x) = (-i)^{\ell} \,H(x)\,  \mathcal{F}_{d+2\ell} [f](x_1, x_2,\ldots, x_d,0,\ldots,0) = s (-i)^{\ell} \,f^{\flat}(x).
\end{equation*}
Note also that $r(\widetilde{P}f^{\flat}) \leq r(Pf)$. The fact that $f^{\flat} \neq {\bf 0}$ follows from the assumption that $f_0\neq {\bf 0}$ and the fact that $H$ cannot be identically zero in any open set of $\R^d$. Then $f^{\flat} \in \mathcal{A}^*_{s(-1)^{(\frak{r(\ell)} + \ell)/2}}\big(\widetilde{P};d\big)$ and \eqref{20200513_16:45} plainly follows.

\subsection{Lifting the dimension} We now work under the additional assumption (P6) for $P$, and consider the case $Q = {\bf 1}$. In case $\ell =0$, we have $H$ being a non-zero constant and \eqref{20200514_10:50} plainly follows. Hence, from now on, let us assume that $1 \leq \ell \leq d$. By a change of variables given by an element $R \in O(d)$ we may assume without loss of generality that $H(x) = x_1x_2\ldots x_{\ell}$. Hence, 
$$\widetilde{P}(x) =  x_1x_2\ldots x_{\ell}\,P_0(|x|) \ \ \ \ (x \in \R^d).$$ 
Let us first argue that we have property (P3) for $\widetilde{P}$. In fact, if $f \in L^1(\R^d)$ is a continuous eigenfunction of the Fourier transform such that 
$$\widetilde{P}(x)f(x) = x_1x_2\ldots x_{\ell}\,P_0(|x|)f(x) = 0  \ \ \ {\rm for } \ \ \ |x| > r,$$
the continuity of $f$ implies that 
$$P_0(|x|)f(x) = 0 \ \ \ {\rm for } \ \ \ |x| > r.$$
Property (P6) holds also for $x \mapsto P_0(|x|)$ ($x \in \R^d$), and we have seen that this implies (P3) for $x \mapsto P_0(|x|)$ ($x \in \R^d$). Hence $f = {\bf 0}$, establishing (P3) for $\widetilde{P}$.

\smallskip

Now let $s' = s(-1)^{(\frak{r(\ell)} + \ell)/2}$ and assume that $\mathcal{A}^{*}_{s'}\big(\widetilde{P};d\big)$ is non-empty. Let $f \in \mathcal{A}^{*}_{s'}\big(\widetilde{P};d\big)$. We start by considering an important reduction.

\subsubsection{Symmetrization with respect to $x_1, x_2, \ldots, x_{\ell}$} Throughout the rest of this proof let us write $x = (x_1, \ldots, x_{\ell}, \x)$ with $\x \in \R^{d-\ell}$. Let
$$w(x_1, x_2, \ldots, x_{\ell}, \x) = f(x_1,x_2, \ldots, x_{\ell}, \x) - f(-x_1, x_2,\ldots, x_{\ell}, \x).$$
Note that $\int_{\R^d} \widetilde{P}w = 2 \int_{\R^d} \widetilde{P}f \leq 0$. Observe also that $w \neq {\bf 0}$, otherwise $\widetilde{P}f$ would be eventually zero and from condition (P3) we would have $f ={\bf 0}$, a contradiction. It is clear that $r(\widetilde{P}w) \leq r(\widetilde{P}f)$, and hence $w \in \mathcal{A}^{*}_{s'}\big(\widetilde{P};d\big)$. Moreover, $w$ is odd with respect to the variable $x_1$. We apply the same symmetrization procedure $\ell - 1$ times, to the variables $x_2, \ldots, x_{\ell}$. One then arrives at a function in $\mathcal{A}^{*}_{s'}\big(\widetilde{P};d\big)$ that is odd with respect to each of the variables $x_1, \ldots, x_{\ell}$ independently.

\smallskip

\noindent {\sc Remark:} As far as radial symmetrization goes, at this point one could proceed as in \eqref{20200415_23:02} and integrate over $SO(d-\ell)$ to symmetrize $f$ with respect to the variable $\x$, but this is not particularly necessary for our argument below.

\subsubsection{Main argument}  \label{Main_Arg}Let us now assume that $f \in \mathcal{A}^{*}_{s'}\big(\widetilde{P};d\big)$ has the symmetries above. Define $g:\R^d \to \R$ by 
\begin{equation}\label{20200407_15:41}
 g(x) = \left\{
 \begin{array}{ll}
      \dfrac{f(x)}{x_1x_2\ldots x_{\ell}}, & {\rm if} \ x_1x_2\ldots x_{\ell} \neq 0; \\
      0, & {\rm if} \ x_1x_2\ldots x_{\ell}=0.
 \end{array}
 \right.  
 \end{equation}
Then $f(x) = x_1x_2\ldots x_{\ell}\, g(x)$ for all $x \in \R^d$, and $g$ is even with respect to each of the variables $x_1, x_2, \ldots, x_{\ell}$ independently. Observe that $P_0(|\cdot|)g$ is eventually non-negative and that 
\begin{equation}\label{20200514_16:16}
r(P_0(|\cdot|)g) = r\big(\widetilde{P}f\big).
\end{equation}

\smallskip

For each $1 \leq k \leq \ell$ let $y_k \in \R^3$ and let $\y \in \R^{d-\ell}$. We now work with the variable $y = (y_1, \ldots, y_{\ell}, \y) \in \R^{d + 2\ell}$. Define the function $\gs:\R^{d+2\ell} \to \R$ by
\begin{equation}\label{20200514_16:13}
\gs(y) = \gs(y_1, \ldots, y_{\ell}, \y) := g\big(|y_1|,\ldots, |y_{\ell}|, \y\big).
\end{equation}
Note that $P\gs$ is eventually non-negative with 
\begin{equation}\label{20200514_16:15}
r(P\gs) = r(P_0(|\cdot|)g).
\end{equation} 
We first observe that $\gs \in L^1(\R^{d + 2\ell})$. In fact, with changes to polar coordinates in each of the first $\ell$ variables on $\R^3$, we get
\begin{align}\label{20200407_17:30}
\begin{split}
\int_{\R^{d + 2\ell}} \big|\gs(y)\big|\,\d y & = \omega_2^{\ell} \int_{\R^{d-\ell}} \int_{(\R^+)^{\ell}} x_1^2 \ldots x_{\ell}^2\,\big|g(x_1, \ldots , x_{\ell}, \y)\big| \, \d x_1\ldots \d x_{\ell}\, \d \y\\
& = \omega_2^{\ell} \int_{\R^{d-\ell}} \int_{(\R^+)^{\ell}} x_1 \ldots x_{\ell}\,\big|f(x_1, \ldots , x_{\ell}, \y) \big|\, \d x_1\ldots \d x_{\ell}\, \d \y\\
& < \infty.
\end{split}
\end{align}
The last integral is finite since $f$ is continuous, $\widetilde{P}f \in L^1(\R^d)$ and we have property (P6) for $P$ (or equivalently, for $P_0$). Similarly, $P\gs \in L^1(\R^{d + 2\ell})$ since
\begin{align}\label{20200415_23:40}
\begin{split}
& \int_{\R^{d + 2\ell}} |P(y)|\,\big|\gs(y)\big|\,\d y \\
&  \ \ \ \ \ = \omega_2^{\ell} \int_{\R^{d-\ell}} \int_{(\R^+)^{\ell}} x_1 \ldots x_{\ell}\,\left|P_0\!\left(\big(x_1^2 +\ldots +x_{\ell}^2 + |\y|^2\big)^{1/2}\right)\, f(x_1, \ldots , x_{\ell}, \y) \right| \d x_1\ldots \d x_{\ell}\, \d \y \\
&   \ \ \ \ \ = \frac{\omega_2^{\ell}}{2^\ell} \int_{\R^d} \big|\widetilde{P}f\big| < \infty.
\end{split}
\end{align}
By recalculating \eqref{20200415_23:40} without the absolute value, and using that $\int_{\R^d} \widetilde{P}f \leq 0$, we find also that $\int_{\R^{d + 2\ell}} P\gs \leq 0$. 

\smallskip

Observe that, for each $1\leq k \leq \ell$, the functions 
\begin{align*}
x_k &\mapsto x_k \,g(x_1, \ldots, x_k,\ldots,x_{\ell}, \x)\\
x_k & \mapsto x_k^2 \,g(x_1, \ldots, x_k,\ldots,x_{\ell}, \x)\\
x_k &\mapsto x_k^2 \,g(x_1, \ldots, x_k,\ldots,x_{\ell}, \x)^2
\end{align*}
are absolutely integrable for a.e. $(x_1, \ldots, x_{k-1}, x_{k+1}, \ldots x_{\ell}, \x) \in \R^{d-1}$ (the second one follows from \eqref{20200407_17:30} and the latter follows from the fact that $f \in L^2(\R^d)$). In the computation below let us denote $x_k^* = (x_k,0,0) \in \R^3$ and $y_k = (y_{k1}, y_{k2}, y_{k3}) \in \R^3$ for $1 \leq k \leq \ell$. By a repeated use of Fubini's theorem and Bochner's relation (Lemma \ref{HB_form}, with $d=\ell=1$ in that statement) we find
\begin{align*}
& s'i^{\frak{r}(\ell)} x_1x_2\ldots x_{\ell}\,\gs (x_1^*, \ldots, x_{\ell}^*, \x) = s'i^{\frak{r}(\ell)} x_1x_2\ldots x_{\ell} \,g(x) = s'i^{\frak{r}(\ell)}  f(x)  = \f(x)\\
& =  \int_{\R^{d-1}}\left(\int_{\R}z_1\, g(z_1,z_2,\ldots, z_{\ell},\widetilde{z})\,e^{-2\pi i x_1z_1}\d z_1\right)z_2\ldots z_{\ell}\, e^{-2\pi i (x_2z_2 + \ldots + x_{\ell}z_{\ell}+\x \cdot \widetilde{z})}\,\d z_2 \ldots \d z_{\ell}\, \d\widetilde{z}\\\
&  =  (-i)\,x_1 \int_{\R^{d-1}}\left(\int_{\R^3} g(|y_1|,z_2,\ldots, z_{\ell},\widetilde{z})\,e^{-2\pi i x_1y_{11}}\d y_1\right)\\
& \ \ \ \ \ \ \ \ \  \ \ \ \ \ \ \ \ \ \ \ \ \ \ \ \ \ \ \ \ \ \ \ \ \ \ \ \ \ \ \ \ \ \ \ \ \ \ \ z_2\ldots z_{\ell}\, e^{-2\pi i (x_2z_2 + \ldots + x_{\ell}z_{\ell}+\x \cdot \widetilde{z})}\,\d z_2 \ldots \d z_{\ell}\, \d\widetilde{z}\\
&  =  (-i)\,x_1 \int_{\R^{d+1}}\left(\int_{\R} z_2 \,g(|y_1|,z_2,\ldots, z_{\ell},\widetilde{z})\,e^{-2\pi i x_2z_{2}}\,\d z_2\right)\\
& \ \ \ \ \ \ \ \ \  \ \ \ \ \ \ \ \ \ \ \ \ \ \ \ \ \ \ \ \ \ \ \ \ \ \ \ \ \ \ \ \ \ \ \ \ \ \ \ z_3\ldots z_{\ell}\, e^{-2\pi i (x_1y_{11} + x_3z_3 + \ldots + x_{\ell}z_{\ell}+\x \cdot \widetilde{z})}\,\d y_1\,\d z_3 \ldots \d z_{\ell}\, \d\widetilde{z}\\
& =  (-i)^2\,x_1x_2 \int_{\R^{d+1}}\left(\int_{\R^3} g(|y_1|,|y_2|,z_3,\ldots, z_{\ell},\widetilde{z})\,e^{-2\pi i x_2y_{21}}\,\d y_2\right)\\
&  \ \ \ \ \ \ \ \ \  \ \ \ \ \ \ \ \ \ \ \ \ \ \ \ \ \ \ \ \ \ \ \ \ \ \ \ \ \ \ \ \ \ \ \ \ \ \ \  z_3\ldots z_{\ell}\, e^{-2\pi i (x_1y_{11} + x_3z_3 + \ldots + x_{\ell}z_{\ell}+\x \cdot \widetilde{z})}\,\d y_1\,\d z_3 \ldots \d z_{\ell}\, \d\widetilde{z}\\
& = \ldots \\
& = (-i)^{\ell}\,x_1\ldots x_{\ell} \int_{\R^{d+2\ell}} g(|y_1|,\ldots, |y_{\ell}|,\widetilde{z}) \, e^{-2\pi i (x_1y_{11} + x_2y_{21} + \ldots +x_{\ell}y_{\ell 1} +\x \cdot \widetilde{z})}\, \d y_1\ldots \d y_{\ell}\, \d\widetilde{z}\\
                & = (-i)^{\ell}\,x_1\ldots x_{\ell} \, \widehat{\gs} (x_1^*, \ldots, x_{\ell}^*, \x).               
\end{align*}
Since $y \mapsto \gs(y)$ is radial on each of the first $\ell$ variables $y_k \in \R^3$, the same is valid for $\widehat{\gs}$ and therefore, if $|y_1|\ldots |y_{\ell}| \neq 0$, we find that 
\begin{equation*}
\widehat{\gs} (y_1, \ldots, y_{\ell}, \y) = s\,\gs  (y_1, \ldots, y_{\ell}, \y).
\end{equation*}
In particular, by the Riemann-Lebesgue lemma, $\gs$ is equal a.e. to a continuous function, that we now call $\overline{\gs}$. All the integrability properties defining the class $\mathcal{A}^{*}_{s}(P;d + 2\ell)$ automatically transfer from $\gs$ to $\overline{\gs}$. We must pay a bit of attention when it comes to \eqref{20200514_16:15}. Note that by definitions \eqref{20200407_15:41} and \eqref{20200514_16:13}, $\gs$ is already continuous on the set $Y = \{y = (y_1, \ldots, y_{\ell}, \y) \in \R^{d + 2\ell} \ : \ |y_1|\ldots|y_{\ell}| \neq 0\}$. So, $\overline{\gs}$ is potentially redefining the values of $\gs$ at the set $Y^c$. We claim that we continue to have
\begin{equation}\label{20200520_21:15}
r\big(P\overline{\gs}\big) = r(P\gs). 
\end{equation}
In fact, let $r = r(P\gs)$. Taking $y \in \R^{d + 2\ell}$ with $|y| >r$, we want to show that $P(y) \overline{\gs}(y) \geq 0$. If $\ |y_1|\ldots|y_{\ell}| \neq 0$ then $P(y) \overline{\gs}(y) = P(y)\gs(y) \geq 0$. If $\ |y_1|\ldots|y_{\ell}| = 0$, we have two options. If $P(y) = 0$ we are done. If not, assume without loss of generality that $P(y) > 0$. In this case, we can take a sequence of points $\{y^{(j)}\}_{j \in \N} \subset Y$ such that $\big|y^{(j)}\big| = |y| >r$ and $y^{(j)} \to y$ as $j \to \infty$. Since $P$ is radial, then $P\big(y^{(j)} \big) \overline{\gs}\big(y^{(j)} \big) = P(y)\overline{\gs}\big(y^{(j)} \big) =  P(y)\gs\big(y^{(j)} \big) = P\big(y^{(j)}\big)\gs\big(y^{(j)} \big) \geq 0$, and we conclude that $\overline{\gs}\big(y^{(j)} \big) \geq 0$ and by continuity $ \overline{\gs}(y) \geq 0$. This shows that $r\big(P\overline{\gs}\big) \leq r(P\gs)$. The reverse inequality is simpler, proceeding along the same lines.

\smallskip

The conclusion is that $\overline{\gs} \in \mathcal{A}^{*}_{s}(P;d + 2\ell)$ and from \eqref{20200514_16:16}, \eqref{20200514_16:15} and \eqref{20200520_21:15} we have
\begin{equation*}
\mathbb{A}^*_{s(-1)^{(\frak{r(\ell)} + \ell)/2}}\big(\widetilde{P};d\big) \geq \mathbb{A}^*_{s}(P;d+2\ell). 
\end{equation*}
This inequality, together with \eqref{20200513_16:45}, leads to the identity \eqref{20200514_10:50}. This concludes the proof.

\smallskip

\noindent{\sc Remark:} It is interesting to notice that if we start with $f \in \mathcal{A}^{*}_{s'}\big(\widetilde{P};d\big)$ as in \S \ref{Main_Arg}, run the procedure of \S \ref{Main_Arg} to arrive at the function $\overline{\gs} \in \mathcal{A}^{*}_{s}(P;d + 2\ell)$, and then run the radialization and dropping procedure of \S \ref{Dim_drop} with this $\overline{\gs}$, we end up with a new function $f_1 = \big(\overline{\gs}\big)^{\flat} \in \mathcal{A}^{*}_{s'}\big(\widetilde{P};d\big)$ that does a job at least as good as the original $f$ and has the form $f_1(x) = x_1\ldots x_{\ell} f_0(x)$, with $f_0$ radial. Such a reduction is not obvious from the start. Since we have explicit radial extremizers for \eqref{20200327_14:00} and \eqref{20200327_14:01} in \cite{CG, CKMRV, Vi}, one can construct explicit extremizers for all the other 14 situations in Corollary \ref{Sharp_const} by formula \eqref{20200409_15:34}.

\section{Power weights: proof of Corollary \ref{Thm_Laplacian}}\label{Sec_Lap}
Although we call this a corollary, it requires a brief proof, that will essentially be a collage of passages from our previous results. For instance, using Theorem \ref{Thm_Harmonic_part} with $H = {\bf 1}$ and $Q = |x|^{\gamma}$ we have: that $\mathcal{A}^{*}_{s}(|x|^{\gamma};d)$ is non-empty for all $\gamma > -d$; that the upper bound in \eqref{20200518_09:54} holds for all $\gamma > -d$ and $s = \pm 1$; and that the identity \eqref{20200518_09:59} holds. From Theorem \ref{Thm_general_sign_uncertainty} (i) we have the existence of extremizers for $\mathbb{A}_{s}(|x|^{\gamma};d)$ when $\gamma \geq 0$, and the fact that they can be taken to be radial follows as in \eqref{20200415_23:02} (from Proposition \ref{Prop_8_integral_zero} we can even assume that $\int_{\R^d}f |x|^{\gamma} = 0$). This leaves us with the task of proving the lower bound in \eqref{20200518_09:54}, which is the actual sign uncertainty principle. We consider below the different regimes.

\subsection{The case $\gamma \geq 0$}\label{Sec6_1} The case $\gamma = 0$ is known (see Theorem A or the remark after Theorem \ref{Thm_general_sign_uncertainty}). Let us assume that $\gamma >0$. Recall that the volume of the unit ball $B = B_1 \subset \R^d$ is given by $|B| = \pi^{d/2} / \Gamma\big(\frac{d}{2} + 1\big)$. Using \eqref{20200416_22:56} and \eqref{20200402_00:50} we find that 
\begin{equation}\label{20200518_16:46}
\mathbb{A}_{s}(|x|^{\gamma};d) \geq \left( \frac{(d + \gamma q')}{d} \frac{1}{|B|} \frac{1}{ \left(2\left( 1 + \frac{\gamma}{d}\right) \left[ \left( 1 + \frac{d}{\gamma}\right) |B|\right]^{\frac{\gamma}{d}}\right)^{q'}}\right)^{\frac{1}{(d + \gamma q')}} = \frac{1}{|B|^{\frac{1}{d}}} \, F(q')\,,
\end{equation}
where 
\begin{equation*}
F(q') = \left(\frac{d + \gamma q'}{d} \right)^{\frac{1}{(d + \gamma q')}} \left(\frac{d}{2(d + \gamma)}\right)^{\frac{q'}{(d + \gamma q')}} \left(\frac{\gamma}{ d + \gamma}\right)^{\frac{\gamma q'}{d(d + \gamma q')}}.
\end{equation*}
Let us briefly indicate why the choice $q' = 1$ indeed maximizes $F(q')$ for all $d$ and $\gamma$ in this case. Write $\gamma = \lambda d$, with $\lambda >0$. Then $\log F(x) = \frac{1}{d} \, H(x)$, with
$$H(x) = \left(\frac{1}{1 + \lambda x}\right)\left( \log(1 + \lambda x) - x \log (2(1 + \lambda)) + \lambda x \log\left(\frac{\lambda}{1 + \lambda}\right)\right).$$
Routine calculus arguments lead to $H'(x) < 0$ for all $x \geq 1$. We then plug $q' = 1$ in \eqref{20200518_16:46} to get 
\begin{equation}\label{20200519_01:00}
\mathbb{A}_{s}(|x|^{\gamma};d)  \geq \left( \frac{\Gamma\big(\frac{d}{2} +1\big)}{\pi^{d/2}}\right)^{\frac{1}{d}} \left(\frac{1}{2}\right)^{\frac{1}{d+ \gamma}}\left(\frac{\gamma}{ d + \gamma}\right)^{\frac{\gamma}{d(d + \gamma)}}.
\end{equation}
This is an explicit lower bound in which the parameters $d$ and $\gamma > 0$ may vary independently. If one is interested in bounds that are uniform on the parameter $\gamma > 0$, we call $x = \gamma/(d + \gamma)$ and note that the function 
$$\gamma \mapsto \left(\frac{1}{2}\right)^{\frac{1}{d+ \gamma}}\left(\frac{\gamma}{ d + \gamma}\right)^{\frac{\gamma}{d(d + \gamma)}}$$
is minimized when $x = 1/2e$, with value $2^{-\frac{1}{d}}\, e^{-\frac{1}{2ed}}$. Then
\begin{equation}\label{20200518_22:34}
\mathbb{A}_{s}(|x|^{\gamma};d)  \geq \left( \frac{\Gamma\big(\frac{d}{2} +1\big)}{2 \,\pi^{d/2} \, e^{\frac{1}{2e}}}\right)^{\frac{1}{d}}.\end{equation}
Using the inequality $\Gamma(x+1) > (\frac{x}{e})^{x} \sqrt{2 \pi x}$ for all $x >0$ we see that the right-hand side of \eqref{20200518_22:34} is greater than $\sqrt{\frac{d}{2 \pi e}}$ for all $d \geq 2$, and we can actually take $c =1$ in \eqref{20200518_09:54}. If $d=1$ then the right-hand side of \eqref{20200518_22:34} is equal to $c \sqrt{\frac{1}{2 \pi e}}$ for $ c = 0.8595\ldots$.

\subsection{The case $-\frac{d}{2} + \varepsilon(d) \leq \gamma < 0$} \label{Sec6.2} Here we use inequality \eqref{20200513_16:45} in Theorem \ref{Dim_Shift} (note that the dimension $d$ here shall correspond to the dimension $d + 2\ell$ in \eqref{20200513_16:45}). If $d = 3$, we let $H(x) = x_1$ (of degree $\ell = 1$). If $d \geq 4$, we let $H$ be a homogeneous and harmonic polynomial of two variables and degree $\ell = -\lfloor \gamma \rfloor$ (e.g. we can take $H(x_1, x_2) = \Re ((x_1 + ix_2)^{\ell})$). Having defined $H$, we let $Q(x) = |x|^{\ell} \cdot \frac{\sgn(H(x))}{H(x)}$ for $H(x) \neq 0$, and zero otherwise. Then \eqref{20200513_16:45} yields
\begin{equation}\label{20200519_00:33}
\mathbb{A}_{s}(|x|^{\gamma};d)  \geq \mathbb{A}_{s(-1)^{(\frak{r(\ell)} + \ell)/2}}\big(\sgn(H(x))|x|^{\gamma +\ell};d-2\ell\big).\
\end{equation}
Note that the final dimension $d- 2\ell$ is at least equal to the number of variables we need to construct our harmonic polynomial $H$ (this is how we define the function $\varepsilon:\N \to \R$), and on the right-hand side of \eqref{20200519_00:33} we now have a homogeneous function $\sgn(H(x))|x|^{\gamma +\ell}$ of degree $0 \leq \gamma + \ell <1$. Note that $|\sgn(H(x))|\,|x|^{\gamma +\ell} = |x|^{\gamma +\ell}$ for a.e. $x \in \R^{d-2\ell}$, and hence the volume of their sub-level sets are the same. We may then proceed as in \S \ref{Sec6_1}, using \eqref{20200416_22:56} and \eqref{20200402_00:50} to arrive at the exact same bounds as in \eqref{20200519_01:00} and \eqref{20200518_22:34}, with $\gamma +\ell$ in the place of $\gamma$ and $d-2\ell$ in the place of $d$. This leads to \eqref{20200518_09:54} in this case.

 \subsection{The case $s =1$ and $-d < \gamma \leq -\frac{d}{2} - \varepsilon(d)$} Recall that for any Schwartz function $f$ we have the identity (see \cite[Chapter V, \S1, Lemma 2]{S})
\begin{equation}\label{20200518_23:36}
\Gamma\left(\frac{d + \gamma}{2}\right)\,\pi^{\frac{-d - \gamma}{2}} \int_{\R^d} |x|^{-d-\gamma} \widehat{f}(x)\, \d x  = \Gamma\left(-\frac{\gamma}{2}\right)\,\pi^{\frac{\gamma}{2}} \int_{\R^d} |x|^{\gamma} f(x)\, \d x.
\end{equation}
Standard approximation arguments show that \eqref{20200518_23:36} remains valid for $f \in L^1(\R^d)$ such that $\widehat{f} \in L^1(\R^d)$. In particular, this implies that 
\begin{equation}\label{20200520_12:00}
\mathbb{A}_{+1}(|x|^{\gamma};d) = \mathbb{A}_{+1}(|x|^{-d-\gamma};d).
\end{equation}
Using this symmetry we fall in the case $-\frac{d}{2} + \varepsilon(d) \leq -d -\gamma < 0$ treated in \S \ref{Sec6.2}. This leads again to \eqref{20200518_09:54} and concludes the proof. 

\smallskip

\noindent{\sc Remark}: The symmetry \eqref{20200520_12:00} is not valid in the case $s = -1$ as we have \eqref{20200518_09:54} and \eqref{20200518_09:59}. In particular, in light of \eqref{20200518_23:36}, this implies that one cannot reduce the search in $\mathbb{A}_{-1}(|x|^{\gamma};d)$, when $\gamma \leq -\frac{d}{2} - \varepsilon(d)$ to functions satisfying $\int_{\R^d}f |x|^{\gamma} = 0$. Proposition \ref{Prop_8_integral_zero} already pointed in this direction.
 
\smallskip

\noindent{\sc Remark}:  Establishing the sign uncertainty for $P(x) = |x|^{\gamma}$, with $-d < \gamma <0$, in a more direct way seems to be subtle. For instance, one could try to prove \eqref{20200331_11:16}, or even a suitable weaker variation of it that would still make the H\"{o}lder's inequality argument in \eqref{20200511_15:24} work. For instance, it would be natural and sufficient to consider an inequality of the type, for $f \in L^1(\R^d)$ with $\f= \pm f$,
\begin{equation}\label{20200512_13:59}
\big\| f |x|^{\alpha}\big\|_q \leq C \,\big\|f |x|^{\gamma}\big\|_1\,,
\end{equation}
for some $\alpha$ and $q$ verifying the conditions: (i) if $q=1$ then $-d < \alpha < \gamma$; or (ii) if $1 < q < \infty$ then $-\frac{d}{q} < \alpha < \gamma + \frac{d}{q'}$. However, the inequality \eqref{20200512_13:59} is simply not true. The following counterexample in dimension $d=1$ was communicated to us by F. Nazarov. Choose a small $\delta>0$ and consider a real-valued, radially non-increasing Schwartz function $g$ supported on $[-\delta,\delta]$ with $g \equiv 1 $ on $[-\delta/2,\delta/2]$.  Let $h_{t}(x) = g(x)\cos(2\pi tx)$ for large $t$ and put $f_t = h_t + \widehat{h_t}$. Then $f_t = \widehat{f_t}$. Noting that $\limsup_{t \to \infty}\big\|\widehat{h_t} |x|^{\gamma}\big\|_1 = 0$, by the triangle inequality,
\begin{equation}\label{20200512_15:09}
\limsup_{t \to \infty}\big\|f_t |x|^{\gamma}\big\|_1 \leq \limsup_{t \to \infty}\big\|h_t |x|^{\gamma}\big\|_1 +  \limsup_{t \to \infty}\big\|\widehat{h_t} |x|^{\gamma}\big\|_1 \leq \big\|g |x|^{\gamma}\big\|_1 \lesssim \delta^{\gamma +1}.
\end{equation}
Similarly, noting that $\limsup_{t \to \infty}\big\|\widehat{h_t} |x|^{\alpha}\big\|_{L^q([-\delta, \delta])} = 0$,
\begin{align}\label{20200512_15:10}
\begin{split}
\liminf_{t \to \infty}\big\|f_t |x|^{\alpha}\big\|_{L^q([-\delta, \delta])}  & \geq 
\liminf_{t \to \infty}\big\|h_t |x|^{\alpha}\big\|_{L^q([-\delta, \delta])} - \limsup_{t \to \infty}\big\|\widehat{h_t} |x|^{\alpha}\big\|_{L^q([-\delta, \delta])}\\
&\geq \liminf_{t \to \infty}\big\|h_t |x|^{\alpha}\big\|_{L^q([-\delta, \delta])}\\
& \gtrsim \delta^{\alpha + \frac{1}{q}}.
\end{split}
\end{align}
If \eqref{20200512_13:59} were true, \eqref{20200512_15:09} and \eqref{20200512_15:10} would imply that $\alpha \geq \gamma + \frac{1}{q'}$, a contradiction. It should be noted that the functions in the counterexample above are eigenfunctions but do not necessarily belong to the class $\mathcal{A}^{*}_{s}(|x|^{\gamma};d)$. Hence, one may still try to find suitable admissibility inequalities like \eqref{20200331_11:16} or \eqref{20200512_13:59} imposing this additional constraint on $f$ (and even assuming that $r(f)$ is small). Several other types of weighted norm inequalities related to uncertainty are considered in \cite{BDell}.

\section*{Acknowledgements}  We would like to thank William Beckner, John Benedetto, Felipe Gon\c{c}alves, Fedor Nazarov and Diogo Oliveira e Silva for lending us their invaluable expertise on the general topic of Fourier uncertainty through insightful remarks. E.C. acknowledges partial support from FAPERJ - Brazil. E.Q-H. acknowledges support from CNPq - Brazil and from the STEP program of ICTP - Italy. We are also thankful to the anonymous referees for the helpful comments.

\end{document}